\documentclass[11pt, a4paper]{amsart}

\pdfoutput=1

\usepackage[english]{babel} 
\usepackage[utf8]{inputenc}  %% les accents dans le fichier.tex

\usepackage{amssymb,amsfonts,amsmath}
\usepackage{amsthm}
\theoremstyle{plain}
\newtheorem{theorem}{Theorem}[section]
\newtheorem{proposition}[theorem]{Proposition}
\newtheorem{corollary}[theorem]{Corollary}
\newtheorem{lemma}[theorem]{Lemma}
\theoremstyle{plain}
\newtheorem{definition}[theorem]{Definition}
\newtheorem{fact}[theorem]{Fact}
\theoremstyle{definition}
\newtheorem{remark}[theorem]{Remark}
\newtheorem{example}[theorem]{Example}

\usepackage[hyphens]{url}
\usepackage{hyperref}
\usepackage{graphicx}
\usepackage{multirow}
\usepackage{multicol}

\newcommand{\mrm}[2]{\multirow{#1}{*}{\ensuremath{#2}}}
\newcommand{\mc}[3]{\multicolumn{#1}{#2}{#3}}
\usepackage{fullpage}
\usepackage{subfigure}

\usepackage[norelsize,onelanguage,ruled,vlined,linesnumbered]{algorithm2e}

\def\leqslant{\leq}
\def\geqslant{\geq}

\newcommand{\F}{{\mathbb F}}
\newcommand{\Fp}{{\F_p}}
\newcommand{\Fps}{{\F_{p^2}}}
\newcommand{\Fpc}{{\F_{p^3}}}

\newcommand{\mF}{{\mathcal F}}
\newcommand{\gq}{\mathfrak q}
\newcommand{\Q}{{\mathbb Q}}
\newcommand{\Z}{{\mathbb Z}}
\newcommand{\ZZ}{{\mathbb Z}}

\newcommand{\EE}{{\mathbb{E}}}

\newcommand{\R}{{\mathbb R}}
\newcommand{\C}{{\mathbb C}}
\newcommand{\OO}{{\mathcal O}}

\newcommand{\GFq}[1]{\mathbb{F}_{#1}}

\newcommand{\GFpn}[2]{\mathbb{F}_{{#1}^{#2}}}

\newcommand{\norm}[1]{{|#1|}}
\DeclareMathOperator{\Gal}{Gal}
\DeclareMathOperator{\GL}{GL}

\DeclareMathOperator{\Disc}{Disc}

\DeclareMathOperator{\Reslt}{Res}

\DeclareMathOperator{\ord}{ord}
\DeclareMathOperator{\Norm}{N}
\DeclareMathOperator{\val}{val}

\DeclareMathOperator{\LLL}{LLL}
\DeclareMathOperator{\GJL}{GJL}
\DeclareMathOperator{\Conj}{Conj}
\DeclareMathOperator{\JLSV}{JLSV}
\DeclareMathOperator{\e}{e} % for exponent

\title{Improvements to the number field sieve for non-prime finite fields \\
\scriptsize{(preliminary version)}}
\author{Razvan Barbulescu \textsuperscript{1,2,3,4}}
\author{Pierrick Gaudry \textsuperscript{1,2,3}}
\author{Aurore Guillevic \textsuperscript{4,3,2}}
\author{François Morain \textsuperscript{4,3,2}}

\thanks{This research was partially funded by Agence
 Nationale de la Recherche grant ANR-12-BS02-001-01.}

\begin{document}
\maketitle
\begin{center}
\small{
  \textsuperscript{1}Université de Lorraine \par
  \textsuperscript{2}Institut national de recherche en informatique et en
automatique (INRIA)\par 
  \textsuperscript{3}Centre national de la recherche scientifique (CNRS)\par 
  \textsuperscript{4}École Polytechnique/LIX\par 
}
\end{center}

\begin{abstract}
    We propose various strategies for improving the computation of
    discrete logarithms in non-prime fields of medium to large
    characteristic using the Number Field Sieve.
    This includes new methods for selecting the
    polynomials; the use of explicit automorphisms; explicit computations
    in the number fields; and prediction that some units have a zero
    virtual logarithm. On the theoretical side, we obtain a new
    complexity bound of $L_{p^n}(1/3,\sqrt[3]{96/9})$ in the medium
    characteristic case. On the practical side, we computed discrete
    logarithms in $\F_{p^2}$ for a prime number $p$ with $80$
    decimal digits.

    {\bf Warning: This unpublished version contains some inexact
    statements.}
\end{abstract}

\section{Introduction}

Discrete logarithm computations in finite fields is one of the
important topics in algorithmic number theory, partly due to its
relevance to public key cryptography. The complexity of discrete
logarithm
algorithms for finite fields $\F_{p^n}$ depends on the size of the
characteristic $p$ with respect to the cardinality $Q=p^n$. In order to
classify the known methods, it is convenient to use the famous $L$
function.
If $\alpha\in[0,1]$ and $c>0$ are two constants, we set 
\begin{equation*}
L_Q(\alpha,c)=\exp\left((c+o(1))(\log Q)^\alpha (\log \log
Q)^{1-\alpha}\right),
\end{equation*}
and sometimes we simply write $L_Q(\alpha)$ if the
constant $c$ is not made explicit. When we consider discrete logarithm
computations, we treat separately families of finite fields for which
the characteristic $p$ can be written in the form $p=L_Q(\alpha)$ for a
given range of values for $\alpha$. We say that we are dealing with
finite fields of {\em small characteristic} if the family is such that
$\alpha < 1/3$; {\em medium characteristic} if we have $1/3<\alpha<2/3$;
and {\em large characteristic} if $\alpha>2/3$. 
In this article, we concentrate on the cases of medium and large
characteristic. This covers also the situation where $p=L_Q(2/3)$, that
we call the medium--large characteristic boundary case.
We start with a brief overview of the
general situation, including the small characteristic case for
completeness (all the complexities mentioned here are based on
unproven heuristics).

The case of small characteristic is the one that has been improved in the
most dramatic way in the recent years. Before 2013, the best known
complexity of $L_Q(1/3, \sqrt[3]{32/9})$ was obtained with the Function
Field Sieve~\cite{Adl94,AdHu99,JoLe02,JoLe06} but a series of 
improvements~\cite{Jou13faster,JouxL14,BaGaJoTh14,GoGrMGZu13,GrKlZuPower2}
has led
to a quasi-polynomial complexity for fixed characteristic, and more
generally to a complexity of $L_Q(\alpha+o(1))$ when $p = L_Q(\alpha)$,
with $\alpha<1/3$.

The case of large characteristic is covered by an
algorithm called the Number Field Sieve (NFS) that is very close to the
algorithm with the same name used for factoring
integers~\cite{LeLe93,Gor93,Schirokauer1993,JoLe02,Sch05}. This
is particularly true for prime fields, and it shares the same complexity
of $L_Q(1/3,\sqrt[3]{64/9})$. In the case of small extension
degrees, the main reference is a variant by Joux, Lercier, Smart and
Vercauteren~\cite{JLSV06} who showed how to get the same complexity in
the whole range of fields of large characteristic.

The case of medium characteristic was also tackled in the same article,
thus getting a complexity of $L_Q(1/3,\sqrt[3]{128/9})$, with another
variant of NFS.

The complexities listed above use versions of NFS where only two number
fields are involved. It is however known that using more number fields can
improve the complexity. For prime fields it has been done
in~\cite{Mat03,CoSe06}, while for large and medium characteristic,
it has been recently studied in~\cite{BarPie2014}.
In all cases, the complexity remains of the form $L_Q(1/3,c)$, but the
exponent constant $c$ is improved: in the large characteristic case we
have $c=\sqrt[3]{(92 + 26\sqrt{13})/27}$, like for prime fields, while in
the medium characteristic case, we have $c=\sqrt[3]{2^{13}/3^6}$.
For the moment, these multiple number
field variants have not been used for practical record computations (they
have not yet been used either for records in integer factorization).

In the medium--large characteristic boundary case, where $p=L_Q(2/3, c_p)$,
the complexity given in~\cite{BarPie2014} is also of the form $L_Q(1/3,c)$,
where $c$ varies between $16/9$ and $\sqrt[3]{2^{13}/3^6}$ in a way that is
non-monotonic with $c_p$. We also mention another variant of NFS that
has been announced~\cite{PiRaTh} that seems to be better in some range of
$c_p$, when using multiple number fields.

In terms of practical record computations, the case of prime fields has
been well studied, with frequent
announcements~\cite{JoLeRecord05,Kle07,DSA180}. In the case of
medium characteristic, there were also some large computations performed
to illustrate the new methods; see Table 8 in~\cite{JL07} and
\cite{Zajac08,hayasaka13}. However in the case
of non-prime field of large characteristic, we are not aware of previous
practical experiments, despite their potential interest in pairing-based
cryptography.
\medskip

\noindent{\bf Summary of contributions.}
Our two main contributions are, on one side, new complexity results for the
finite fields of medium characteristic, and on the other side, a practical
record computation in a finite field of the form $\F_{p^2}$. 

Key tools for these results are two new methods for selecting the number
fields; the first one is a generalization of the method by Joux and
Lercier~\cite{JoLe03} and we call the second one the conjugation method.
It turned out that both of them have practical and theoretical
advantages.

On the theoretical side, the norms that must be tested for smoothness
during NFS based on the conjugation method or the generalized
Joux-Lercier method are smaller than the ones obtained with previous
methods for certain kind of finite fields. Therefore, the probability of
being smooth is higher, which translates into a better complexity.
Depending on the type of finite fields, the gain
is different:
\begin{itemize}
    \item In the medium characteristic finite fields, NFS with the conjugation
        method has a complexity of $L_Q(1/3, \sqrt[3]{96/9})$. This is
        much better than the complexity of $L_Q(1/3, \sqrt[3]{128/9})$
        obtained in~\cite{JLSV06} and also beats the $L_Q(1/3,
        \sqrt[3]{2^{13}/3^6})$ complexity of the multiple number
        field algorithm of~\cite{BarPie2014}.
    \item In the medium--large characteristic boundary case, the
        situation is more complicated, but there are also families of
        finite fields for which the best known complexity is obtained
        with the conjugation method or with the generalized Joux-Lercier
        method. The overall minimal complexity is obtained for fields
        with $p=L_Q(2/3, \sqrt[3]{12})$, where the complexity drops to
        $L_Q(1/3,\sqrt[3]{48/9})$ with the conjugation method.
\end{itemize}

On the practical side, the two polynomials generated by the conjugation
method (and for one of the polynomials with the generalized Joux-Lercier
construction) enjoy structural properties: it is often possible to use
computations with explicit units (as was done in the early ages of NFS
for factoring, before Adleman introduced the use of characters), thus
saving the use of Schirokauer maps that have a non-negligible cost during
the linear algebra phase. Furthermore, it is also often possible to
impose the presence of field automorphisms which can be used to speed-up
various stages of NFS, as shown in~\cite{JLSV06}.

Finally, the presence of automorphisms can interact with the
general NFS construction and lead to several units having zero virtual
logarithms. This is again very interesting in practice, because some
dense columns (explicit units or Schirokauer maps) can be erased in the
matrix. A careful study of this phenomenon allowed us to predict
precisely when it occurs.

All these practical improvements do not change the complexity but make
the computations faster. In fact, even though the conjugation method is
at its best for medium characteristic, it proved to be competitive even
for quadratic extensions. It was therefore used in our record computation
of discrete logarithm in the finite field $\F_{p^2}$ for a random-looking
prime $p$ of 80 decimal digits.  The running time was much less than what
is required to solve the discrete logarithm problem in a prime field of
similar size, namely 160 decimal digits.

\medskip

\noindent{\bf Outline.} In Section~\ref{sec:refresher} we make a quick
presentation of NFS, and we insist on making precise the definitions of
virtual logarithms in the case of explicit units and in the case of
Schirokauer maps. In Section~\ref{sec:galois} we show how to
obtain a practical improvement using field automorphisms, again taking
care of the two ways of dealing with units. Then, in
Section~\ref{sec:vanishing} we explain how to predict the cases where the
virtual logarithm of a unit is zero, and in Section~\ref{sec:units} we
show how to use this knowledge to reduce the number of Schirokauer maps if we
do not use explicit units. Finally, in Section~\ref{sec:polyselect} we
present our two new methods for selecting polynomials, the complexities of
which are analyzed in Section~\ref{sec:complexity}. We conclude in
Section~\ref{sec:effective} with a report about our practical computation
in $\F_{p^2}$.

\section{The number field sieve and virtual logarithms}
\label{sec:refresher} 
\subsection{Sketch of the number field sieve algorithm}

In a nutshell, the number field sieve for discrete logarithms in
$\F_{p^n}$ is as follows. In the first stage, called polynomial selection,
two polynomials $f,g$ in $\Z[x]$ are constructed (we assume that $\deg f
\geqslant \deg g$), such that their reductions modulo $p$ have a
common monic irreducible
factor $\varphi_0$ of degree $n$. For simplicity, we assume that $f$ and
$g$ are monic. 
We call $\varphi$ a monic
polynomial of $\Z[x]$ whose reduction modulo $p$ equals $\varphi_0$. Let
$\alpha$ and $\beta$ be algebraic numbers such that $f(\alpha)=0$ and
$g(\beta)=0$ and let $m$ be a root of $\varphi_0$ in $\F_{p^n}$, allowing us to
write $\F_{p^n}=\F_p(m)$. Let $K_f$ and $K_g$ be the number fields
associated to $f$ and $g$ respectively, and $\OO_f$ and $\OO_g$ their
rings of integers.
\smallskip

For the second stage of NFS, called relation collection or sieve, a
smoothness bound $B$ is chosen and we consider the associated 
factor base
\begin{equation*}
\mF=\{\text{prime ideals $\gq$ in $\OO_f$ and $\OO_g$
of norm less than }B\},
\end{equation*}
that we decompose into $\mF = \mF_f \cup \mF_g$ according to the ring of
integers to which the ideals belong.
An integer is $B$-smooth if all its prime factors are less
than $B$. For any polynomial $\phi(x)\in\Z[x]$, the algebraic integer
$\phi(\alpha)$ (resp. $\phi(\beta))$) in $K_f$ (resp. $K_g$) is
$B$-smooth if the corresponding principal ideal $\phi(\alpha)\OO_f$ (resp.
$\phi(\beta)\OO_g$) factors
into prime ideals that belong to $\mF_f$ (resp. $\mF_g$).
This is almost, but not exactly equivalent to asking that the norm 
$\Reslt(\phi,f)$ (resp. $\Reslt(\phi,g)$) is $B$-smooth.

In the sieve stage, one collects $\#\mF$ polynomials $\phi(x)\in\Z[x]$ with
coprime coefficients and degree bounded by $t-1$, for a parameter $t\geq 2$
to be chosen, such that both $\phi(\alpha)$ and $\phi(\beta)$ are
$B$-smooth, so that we get {\em relations} of the form:
\begin{equation} \label{eq:doubly smooth}
    \left\{ \begin{array}{l}
\phi(\alpha)\OO_f=\prod_{\gq\in\mF_f}\gq^{\val_\gq\left(\phi(\alpha)\right)}\\
\phi(\beta)\OO_g=\prod_{\mathfrak{r}\in\mF_g}
    \mathfrak{r}^{\val_\mathfrak{r}\left(\phi(\beta)\right)}.\\
\end{array} \right.
\end{equation}
The norm of $\phi(\alpha)$ (resp. of $\phi(\beta)$) is
the product of the norms of the ideals in the right hand side and will be
(crudely) bounded by the size of the finite field; therefore
the number of ideals involved in a relation is less than $\log_2 (p^n)$.
One can also remark that the ideals that can occur in a relation have
degrees that are at most equal to the degree of $\phi$, that is $t-1$.
Therefore, it makes sense to include in $\mF$ only the ideals of degree
at most $t-1$ (for a theoretical analysis of NFS one can consider the
variant where only ideals of degree one are included in the factor base).

In order to estimate the probability to get a relation for a polynomial
$\phi$ with given degree and size of coefficients, we make 
the common heuristic that the integer $\Reslt(\phi,f)\cdot \Reslt(\phi,g)$ has the
same probability to be $B$-smooth as a random integer of the same size
and that the bias due to powers is negligible. Therefore, reducing the
expected size of this product of norms is the main criterion when
selecting the polynomials $f$ and $g$.
\smallskip

In the linear algebra stage, each relation is rewritten as a linear
equation between the so-called virtual logarithms of the factor base
elements. We recall this notion in Section~\ref{ssec:virtual logarithms}.
We make the usual heuristic that this system has a space of solutions of
dimension one. Since the system is sparse, an iterative algorithm like 
Wiedemann's~\cite{Wiedemann1986} is used to compute 
a non-zero solution in quasi-quadratic time. This gives the (virtual)
logarithms of all the factor base elements.

In principle, the coefficient ring of the
matrix is $\Z/(p^n-1)\Z$, but it is enough to solve it modulo each prime
divisor $\ell$ of $p^n-1$ and then to recombine the results using the
Pohlig-Hellman algorithm~\cite{PohligHellman1978}. Since one can use
Pollard's method~\cite{Pol78} for small primes $\ell$, we can suppose that
$\ell$ is larger than $L_{p^n}(1/3)$. It allows us then to assume that $\ell$ is
coprime to $\Disc(f)$, $\Disc(g)$, the class numbers of $K_f$ and $K_g$,
and the orders of the roots of unity in $K_f$ and $K_g$. These
assumptions are used in many places in the rest of the article, sometimes
implicitly.
\smallskip

In the last stage of the algorithm, called individual logarithm, the
discrete logarithm of any element $z=\sum_{i=0}^{n-1}z_i m^i$ of
$\F_{p^n}$ in the finite field is computed.  For this, we associate to
$z$ the algebraic number $\overline{z}=\sum_{i=0}^{n-1}z_i\alpha^i$ in
$K_f$ and check whether the corresponding principal ideal factors into
prime ideals of norms bounded by a quantity $B'$ larger than $B$. We also
ask the prime ideals to be of degree at most $t-1$. If $\overline{z}$
does not verify these smoothness assumptions, then we replace $z$ by
$z^e$ for a randomly chosen integer $e$ and try again. 
This allows to obtain a linear equation similar to those of the
linear system, in which one of the unknowns is $\log z$. The second step
of the individual logarithm stage consists in obtaining relations between a
prime ideal and prime ideals of smaller norm, until all the ideals involved
are in $\mF$. This allows to backtrack and obtain $\log z$.

\subsection{Virtual logarithms}
\label{ssec:virtual logarithms} 

In this section, we recall the definition of virtual logarithms, while
keeping in mind that in the rest of the article, we are going to use
either explicit unit computations or Schirokauer maps.
The constructions work independently in each number field, so we explain
them for the field $K_f$ corresponding to the polynomial $f$. During NFS,
this is also applied to $K_g$.

We start by fixing a notation for the ``reduction modulo $p$'' map that
will be used in several places of the article.

\begin{definition}[Reduction map]
    Let $\rho_f$ be the map from $\OO_f$ to $\F_{p^n}$ defined by the
    reduction modulo the prime ideal $\mathfrak{p}$ above $p$ that
    corresponds to the factor $\varphi$ of $f$ modulo $p$. This is a ring
    homomorphism. Furthermore, if the norm of $z$ is coprime to $p$, then
    $\rho_f(z)$ is non-zero in $\F_{p^n}$. We can therefore extend
    $\rho_f$ to the set of elements of $K_f$ whose norm has a
    non-negative valuation at $p$.
    
    Since in this article we will often consider the discrete logarithm
    of the images by $\rho_f$, we restrict its definition to the elements
    of $K_f$ whose norm is coprime to $p$, for which the image is
    non-zero.
\end{definition}

Let $h$ be the class number $K_f$ that we assume to be coprime to the
prime $\ell$ modulo which the logarithms are computed. We also need to
consider the group of units $U_f$ in $\OO_f$. By Dirichlet's theorem it is
a finitely generated abelian group of the form
  $$U_f \sim U_{tors} \times \Z^{r},$$
where $r$ is the unit rank given by $r = r_1 + r_2 - 1$ where
$r_1$ is the number of real roots of $f$ and $2 r_2$ the number of
complex roots, and $U_{tors}$ is cyclic. Any unit $\eta \in U_f$ can be
written
$$\eta = \varepsilon_0^{u_0} \prod_{j=1}^{r} \varepsilon_j^{u_j}$$
for {\em fundamental units} $\varepsilon_j$, $j \geq 1$, and
$\varepsilon_0$ a root of unity.

For each prime ideal $\gq$ in the factor base $\mF_f$, the ideal $\gq^h$
is principal and therefore there exists a generator $\gamma_\gq$ for it.
It is not at all unique, and the definition of the virtual logarithms
will depend on the choice of the fundamental units and of the set of
generators for all the ideals of $\mF_f$. We denote by $\Gamma$ this
choice, and will use it as a subscript in our notations to remember the
dependence in $\Gamma$. In particular, the notation $\log_\Gamma$ used
just below means that the definition of the virtual logarithm depends on
the choice of $\Gamma$, and does not mean that the logarithm is given in
base $\Gamma$; in fact all along the article we do not make explicit the
generator used as a basis for the logarithm in the finite field.

\begin{definition}[Virtual logarithms -- explicit version]
    Let $\gq$ be an ideal in the factor base $\mF_f$, and $\gamma_\gq$
    the generator for its $h$-th power, given by the choice $\Gamma$.
    Then the virtual logarithm of $\gq$ w.r.t. $\Gamma$ is given by
    \begin{equation*}
        \log_\Gamma\gq\equiv h^{-1}\log(\rho_f(\gamma_\gq)) \mod \ell,
    \end{equation*}
    where the $\log$ notation on the right-hand side is the discrete
    logarithm function in $\F_{p^n}$.

    In the same manner, we define the virtual logarithms of the
    units by
    \begin{equation*} 
        \log_\Gamma\varepsilon_j \equiv
        h^{-1}\log(\rho_f(\varepsilon_j)) \mod \ell.
    \end{equation*}
\end{definition}

We now use this definition to show that for any polynomial $\phi$
yielding a relation, we can obtain a linear expression between the logarithm of
$\rho_f(\phi(\alpha))$ in the finite field and the virtual logarithms of the ideals
involved in the factorization of the ideal $\phi(\alpha)\OO_f$:
$$
\phi(\alpha)\OO_f=\prod_{\gq\in\mF_f} \gq^{\val_{\gq}\left( \phi(\alpha)
\right)}.
$$
After raising this equation to the power $h$, we get an equation between
principal ideals that can be rewritten as the following equation between field
elements:
$$
\phi(\alpha)^h=\varepsilon_0^{u_{\phi, 0}}
               \prod_{j=1,r}\varepsilon_j^{u_{\phi,j}}
               \prod_{\gq\in\mF_f}
                  \gamma_\gq^{\val_{\gq}\left( \phi(\alpha) \right)},
$$
where the $u_{\phi,j}$ are integers used to express the unit that pops up
in the process. We then apply the map $\rho_f$, and use 
the fact that it is an homomorphism. We obtain therefore
$$
\rho_f(\phi)^h = \rho_f(\varepsilon_0)^{u_{\phi,0}}\prod_{j=1,r}
\rho_f(\varepsilon_j)^{u_{\phi,j}}\prod_{\gq\in\mF_f}
           \rho_f(\gamma_\gq)^{\val_{\gq}\left( \phi(\alpha) \right)},
$$
from which we deduce our target equation by taking logarithms on both
sides:
\begin{equation}
\label{eq:explicit}
\log\left(\rho_f(\phi(\alpha))\right) \equiv 
      \sum_{j=1}^ru_{\phi,j}\log_\Gamma\varepsilon_j +
      \sum_{\gq\in\mF_f}\val_\gq\left(\phi(\alpha)\right) \log_\Gamma\gq
      \mod \ell.
\end{equation}
In this last step, the contribution of the root of unity $\varepsilon_0$
has disappeared. Indeed, the following simple lemma states that its
logarithm vanishes modulo $\ell$.
\begin{lemma} \label{lem:roots of unity}
    Let $\varepsilon_0$ be a torsion unit of order $r_0$ and assume that
    $\gcd(hr_0, \ell) = 1$. Then we have $\log_\Gamma\varepsilon_0 \equiv 0
    \mod \ell$.
\end{lemma}

\begin{proof}
    Since $\varepsilon_0^{r_0}=1$ in $K_f$, we have
    $\rho_f(\varepsilon_0)^{r_0} = 1$ in
    $\F_{p^n}$ and we get
    $hr_0 \log_\Gamma \varepsilon_0 \equiv 0 \mod \ell$.
\end{proof}

In order to make the equation~\ref{eq:explicit} explicit for a given
$\phi$ that yields a relation, it is necessary to compute the class
number $h$ of $K_f$, to find the generators of all the $\gq^h$ and to
compute a set of fundamental units. These are reknowned to be difficult
problems except for polynomials $f$ with tiny coefficients.

We now recall an alternate definition of virtual logarithms based on
the so-called Schirokauer maps, for which none of the above have to be
computed explicitly.

\begin{definition}[Schirokauer maps]
    Let $K_\ell$ be the multiplicative subgroup of $K_f^*$ of elements
    whose norms are coprime to $\ell$.

    A Schirokauer map is an application
    $\Lambda:(K_\ell)/(K_\ell)^\ell\rightarrow (\Z/\ell\Z)^{r}$
    such that
    \begin{itemize}
        \item $\Lambda(\gamma_1\gamma_2)=\Lambda(\gamma_1)+\Lambda(\gamma_2)$
    ($\Lambda$ is linear);
        \item $\Lambda(U_f)$ is surjective
            ($\Lambda$ preserves the unit rank).
    \end{itemize}
\end{definition}

Schirokauer~\cite{Schirokauer1993} proposed a fast-to-evaluate map
satisfying these conditions that we recall now.
Let us define first an integer, that is the LCM of the exponents required
to apply Fermat's theorem in each residue field modulo $\ell$:
$$\epsilon=\text{lcm}\{\ell^\delta-1,\ \text{such that}\ f(x)\bmod
    \ell\ \text{has an irreducible factor of degree }\delta\}.$$
Then, by construction, for any element $\gamma$ in $K_\ell$, we have
$\gamma^\epsilon$ congruent to $1$ in all the residue fields above
$\ell$. Therefore, the map
\begin{equation}\label{eq:polschi}
\gamma(\alpha) \mapsto \frac{\gamma(x)^\epsilon-1}{\ell}\bmod (\ell, f(x)),
\end{equation}
is well defined for $\gamma\in K_\ell$. Taking the coordinates of the
image of this map in the basis $1,X,\ldots,X^{\deg f-1}$, we can expect
to find $r$ independent linear combinations of these coordinates. They
then form a Schirokauer map.
In~\cite{Sch05}, Schirokauer gave heuristic arguments for the existence
of such independent linear combinations; and in practice, in most of the
cases, taking the $r$ first coordinates is enough.
\smallskip

From now on, we work with a fixed choice of Schirokauer map that we
denote by $\Lambda$. We start by taking another set of $r$ independent
units: for each $j\in[1,r]$, we choose a unit $\varepsilon_j$ such that
$$\Lambda(\varepsilon_j)=(0,\ldots,0,h,0,\ldots,0),$$
where the coordinate $h$ is in the $j$-th position.
We can then refine the choice of the generators of
the $h$-th power of the factor base ideals, so that we get another
definition of the virtual logarithms.

\begin{definition}[Virtual logarithms -- Schirokauer's version]
    Let $\Lambda$ be a Schirokauer map as described above.
    Let $\gq$ be an ideal in the factor base $\mF_f$, and $\gamma_\gq$
    an (implicit) generator for its $h$-th power, such that $\Lambda(\gamma_\gq) = 0$.
    Then the virtual logarithm of $\gq$ w.r.t. $\Lambda$ is given by
    \begin{equation*}
        \log_\Lambda\gq\equiv h^{-1}\log(\rho_f(\gamma_\gq)) \mod \ell.
    \end{equation*}
    The virtual logarithms of the units are defined in a similar manner:
    \begin{equation*} 
        \log_\Lambda\varepsilon_j \equiv
        h^{-1}\log(\rho_f(\varepsilon_j)) \mod \ell.
    \end{equation*}
\end{definition}

As shown in~\cite{Sch05}, by an argument similar to the
case of explicit generators, one can write
\begin{equation}
\label{eq:implicit}
\log\left(\rho_f(\phi(\alpha))\right)\equiv
    \sum_{j=1}^r\lambda_j\left(\phi(\alpha)\right)\log_\Lambda\varepsilon_j
   +\sum_{\gq\in\mF_f}\val_\gq\left(\phi(\alpha)\right)
              \log_\Lambda\gq\mod \ell,
\end{equation}
where $\lambda_j$ is the $j$-th coordinate of $\Lambda$.

\subsection{Explicit units or Schirokauer maps?}

Equation~\eqref{eq:implicit} can be written for the two polynomials $f$
and $g$ and hence we obtain a linear equation relating only virtual
logarithms. We remark that it is completely allowed to use the virtual
logarithms in their explicit version for one of the polynomials if it is
feasible, while using Schirokauer maps on the other side.

Using explicit units requires to compute a generator for each
ideal in the factor base, and therefore the polynomial must have small
coefficients (and small class number). A lot of techniques and algorithms
are well described in \cite{LeLe93}. These include generating units and
generators in some box or ellipsoid of small lengths, and recovery of
units using floating point computations. These are quite easy to
implement and are fast in practice. We may do some simplifications when
$K_f$ has non-trivial automorphisms, since in this case the generators of
several ideals can be computed from one another using automorphisms (see
Section \ref{sec:galois}).

In the general case, one uses
Schirokauer maps whose coefficients are elements of $\Z/\ell\Z$ for a
large prime $\ell$. In our experiments, the values of the Schirokauer
maps seem to spread in the full range $[0,\ell-1]$ and must be stored
on $\log_2\ell$ bits. In a recent record~\cite{DSA180}, each row of
the matrix consisted in average of $100$ non-zero entries in the
interval $[-10,10]$ and two values in $[0,\ell-1]$, for a prime $\ell$
of several machine words. It is then worth to make additional
computations in order to reduce the number of Schirokauer maps. This
motivated our study in Section \ref{sec:units}.

\section{Exploiting automorphisms}
\label{sec:galois}
Using automorphisms of the fields involved in a discrete logarithm
computation is far from being a new idea. It was already proposed by
Joux, Lercier, Smart and Vercauteren~\cite{JLSV06} and was a key
ingredient in many of the recent record computations in small
characteristic~\cite{JouxRecord13,GrKlZu14}.
In this section we recall the basic idea and make explicit the
interaction with both definitions of virtual logarithms, using or not
Schirokauer maps.

\subsection{Writing Galois relations} 

The results of this subsection apply potentially to both number fields
$K_f$ and $K_g$ independently. Therefore, we will express all the
statements with the notations corresponding to the polynomial $f$ (that
again, we assume to be monic for simplicity).

We assume that $K_f$ has an automorphism $\sigma$, and we denote by
$A_\sigma$ and $A_{\sigma^{-1}}$ the
polynomials of $\Q[x]$ such that $\sigma(\alpha)=A_\sigma(\alpha)$ and
$\sigma^{-1}(\alpha)=A_{\sigma^{-1}}(\alpha)$.  For any
subset $I$ of $K_f$, we denote by $I^\sigma$ the set $\{\sigma(x)\mid x\in
I\}$.

\begin{proposition} Let $q$ be a rational prime not dividing the index
$\left[\OO:\Z[\alpha]\right]$ of the polynomial~$f$. Then, any prime ideal
above $q$ of degree one can be generated by two elements of the form
$I=\langle q,\alpha-r\rangle$ for
some root $r$ of $f$ modulo $q$. If the denominators of the
coefficients of $A_\sigma$ and $A_{\sigma^{-1}}$ are not divisible by $q$,
then we have \begin{equation*} I^\sigma=\left\langle
q,\alpha-A_{\sigma^{-1}}(r)\right\rangle.  \end{equation*}
\end{proposition} \begin{proof} Since $\sigma^{-1}$ is an automorphism, we
have $f\left( A_{\sigma^{-1}}(\alpha) \right)=0$. This is equivalent to
$f(A_{\sigma^{-1}}(x))\equiv 0 (\bmod f(x))$ and then
$f(A_{\sigma^{-1}}(x))=u(x) f(x)$ for some polynomial $u\in\Q[x]$. By
evaluating in $r$ we obtain $f(A_{\sigma^{-1}}(r))\equiv 0\pmod q$. Then,
by Dedekind's Theorem, $J=\left\langle
q,\alpha-A_{\sigma^{-1}}(r)\right\rangle$ is a prime ideal of degree one.

Since $q$ and $A_{\sigma^{-1}}(r)$ are rational, we have
$J^{\sigma^{-1}}=\langle q,A_{\sigma^{-1}}(\alpha)-A_{\sigma^{-1}}(r)\rangle$.
Since the polynomial $ A_{\sigma^{-1}}(x)-A_{\sigma^{-1}}(r)$ is divisible
by $x-r$, $J^{\sigma^{-1}}$ belongs to $\langle q,\alpha-r\rangle=I$. Therefore, $J$ belongs to
$I^\sigma$. But $J$ is prime, so $J=I^\sigma$.

\end{proof}

Before stating the main result on the action of $\sigma$ on the virtual
logarithms, we need the following result on the Schirokauer maps.

\begin{lemma}\label{lem:kernel} 
    Let $\Lambda$ be a Schirokauer map modulo $\ell$ associated to $K_f$ and
    let $\sigma$ be an automorphism of $K_f$. Assume in addition that 
    this Schirokauer map is based on the construction of
    Equation~\ref{eq:polschi}.
    Then we have
    \[\ker \Lambda= \ker (\Lambda\circ \sigma).\]
\end{lemma}

\begin{proof}
Let $A_\sigma(x)\in\Z[x]$ be such that $A_\sigma(\alpha)=\sigma(\alpha)$.
If $\gamma=P(\alpha)$ is in the kernel of $\Lambda$, then there exist
$u,v\in \Z[x]$ such that
\begin{equation}
P(x)^\epsilon-1=\ell^2u(x)+\ell v(x) f(x). 
\end{equation}
By substituting $A_\sigma(x)$ to $x$, we obtain
\begin{equation}
P(A_\sigma(x))^\epsilon-1=\ell^2u(A_\sigma(x))+\ell v(A_\sigma(x)) f(A_\sigma(x)). 
\end{equation}
 Since $\sigma$ is an automorphism of $f$, $f(A_\sigma(x))$ is a multiple
of $f(x)$. Hence, we obtain that $\sigma(\gamma)=P(A_\sigma(\alpha))$ is in the kernel of
$\Lambda$. 
\end{proof}

\begin{example}
When $f$ is an even polynomial, i.e.~ $f(-x)=f(x)$, the application
$\sigma(x)=-x$ is an automorphism of the number field $K_f=\Q[x]/f(x)$.
Consider the Schirokauer map as defined in Equation~\ref{eq:polschi}.
We denote by $\Lambda=(\lambda_1,\ldots,\lambda_r)$ the $r$ first
coordinates in basis $1,X,\ldots,X^{\deg f-1}$, and we assume that they
are independent, so that $\Lambda$ is indeed a Schirokauer map.
Then applying the automorphism, we get
$\Lambda\circ\sigma=(\lambda_1,-\lambda_2,\lambda_3,-\lambda_4,\ldots,
(-1)^{r+1}\lambda_r)$, and we can check that its kernel coincides with
the kernel of $\Lambda$.
\end{example}

The following counter-example shows that the condition that $\Lambda$ is
constructed from Equation~\ref{eq:polschi} is necessary for
Lemma~\ref{lem:kernel} to hold.

\begin{example}
Let $\Lambda=(\lambda_1,\ldots,\lambda_r)$ be a Schirokauer map of $K_f$
with respect to $\ell$, $\sigma$ an automorphism of $K_f$ and $\gq$ a prime
ideal. Then
$\Lambda'=(\lambda_1+
\val_{\gq}(\cdot),\lambda_2,\lambda_3,\ldots,\lambda_r)$
does not satisfy $\ker \Lambda'=\ker \Lambda'\circ \sigma$.
Indeed, let $\gamma$ be a generator of $(\gq^{\sigma^{-1}})^h$ with
$\Lambda(\gamma)=0$. On the one hand we have $\Lambda'(\gamma)=0$. On the
other hand, the first coordinate of $\Lambda'(\sigma(\gamma))$ is the
valuation in $\gq$ of $\sigma(\gamma)$, which is non zero because
$\sigma(\gamma)$ is in $\gq$.  
\end{example}

\begin{theorem}[Galois relations] \label{th:galois}
    We keep the same notations as above,
    where in particular $\varphi$ is a degree-$n$ irreducible factor of
    $f$ modulo $p$. 
    Let $\sigma$ be an automorphism of $K_f$ different from the identity
    such that \[\varphi(\rho_f(A_\sigma(\alpha)))=0.\]
	Then, there exists a constant  $\kappa\in[1,\ord(\sigma)-1]$ such that the following holds:
    \begin{enumerate}
        \item Let $\Gamma$ be a choice of explicit generators that is
            compatible with $\sigma$, i.e.~such that for any prime
            ideal $\gq$ the generators for the $h$-th powers of $\gq$ and
            $\sigma(\gq)$ are conjugates:
            $$ \gamma_{\sigma(\gq)} = \sigma(\gamma_\gq).$$
            Then we have for any prime ideal $\gq$:
\[\log_\Gamma \gq^\sigma \equiv p^\kappa \log_\Gamma \gq\pmod \ell.\]

        \item For any Schirokauer map $\Lambda$ which has a polynomial
	formula (as in Lemma~\ref{lem:kernel}) and for any prime ideal $\gq$, we
	have \[\log_\Lambda \gq^\sigma \equiv p^\kappa \log_\Lambda \gq\pmod\ell.\]
    \end{enumerate}
\end{theorem}

\begin{proof}

Since $\rho_f(\sigma(\alpha))$ is a root of $\varphi$ other than
$m=\rho_f(\alpha)$, the map
$T(x)\mapsto T(A_\sigma(x))$ is an element of $\Gal(\F_{p^n}/\F_p)$ other
than the identity.  So, there exists a constant
$\kappa\in[1,\ord(\sigma)-1]$ such that $A_\sigma(x)=x^{p^\kappa}$ for all
$x\in\F_{p^n}$. In particular, if $\gq$ is a prime ideal and
$\gamma_\gq$ is a generator of $\gq^h$, we have
\begin{equation} \label{eq:p^k}
    \log \rho_f(\sigma(\gamma_\gq)) = p^\kappa \log(\rho_f(\gamma_\gq)).
\end{equation}

In the first assertion of the theorem, it is assumed that
$\sigma(\gamma_\gq)$ is precisely the generator used for $\sigma(\gq)^h$,
and therefore the relation between virtual logarithms follows from their
definition.

For the second assertion, the compatibility of the generators is deduced
from the definition of the virtual logarithms using Schirokauer maps.
Indeed, for any prime ideal $\gq$, the generator used for the definition
of $\log_\Lambda\gq$ is such that $\Lambda(\gamma_\gq)=0$. By
Lemma~\ref{lem:kernel}, $\gamma_\gq$ is also in the kernel of
$\Lambda\circ\sigma$, that is $\Lambda(\sigma(\gamma_\gq))=0$, so that
the conjugate of the generator is a valid generator for the conjugate of
the ideal. The conclusion follows.
\end{proof}

We give an immediate application of the preceding results,
which is useful when $K_f$ is an imaginary quadratic field.
\begin{lemma}
Let $q$ be a rational prime which is totally ramified in $K_f$,
and write $q\, \OO_f = \mathfrak{q}^n$. Assume
that the unit rank of $K_f$ is $0$ and that $n$ is coprime to $\ell$. Then we
have 
\begin{equation*}
 \log \mathfrak{q} \equiv 0 \bmod \ell.
\end{equation*} 
\end{lemma} 

\begin{proof}
Let $h$ be the class number of
$K_f$ and $\gamma_\mathfrak{q}$ a generator of $\mathfrak{q}^h$ such that
$\log \mathfrak{q}=h^{-1}\log \gamma_{\mathfrak{q}}$. Then one can write
$q^h=u (\gamma_{\mathfrak{q}})^n$ for some root of unity $u$. By
Lemma~\ref{lem:roots of unity}, $\log u\equiv 0\mod \ell$, so
\begin{eqnarray*}
\log (q^h) &\equiv& \log ((\gamma_{\mathfrak{q}})^n)
\bmod \ell.
\end{eqnarray*}
Since $q$ belongs to the subgroup of
$\F_{q^n}$ given by equation $x^{q-1}=1$ and since $\gcd(q-1,\ell)=1$,
Lemma~\ref{lem:roots of unity} gives $\log q=0$. Then the results follows
from the fact that $n$ is coprime to $\ell$.
\end{proof}

\subsection{Using Galois relations in NFS}

Let $\sigma$ and $\tau$ be automorphisms of $K_f$ and $K_g$, and let us
assume that they verify the hypothesis of Theorem~\ref{th:galois}.  We
can split $\mF_f$ and $\mF_g$ respectively in orbits
$(\gq,\gq^\sigma,\ldots)$ if $\gq$ is in $\mF_f$ and
$(\gq,\gq^\tau,\ldots)$ if $\gq$ is in $\mF_g$.

This allows to reduce the number of unknowns in the linear algebra stage
by a factor $\ord(\sigma)$ on the $f$-side and by a factor $\ord(\tau)$
on the $g$-side, at the price of having entries in the matrix that are
roots of unity modulo $\ell$ instead of small integers.
We collect as many relations as unknowns, hence reducing
also the cost of the sieve. 

Note that the case where $\sigma$ or $\tau$ is the identity is not
excluded in our discussion (in that case, the orbits are singletons on
the corresponding side).
\smallskip

As an example, in Section~\ref{sec:polyselect} we will see
how to construct polynomials $f$ and $g$ whose number fields have automorphisms
$\sigma$ and $\tau$, both of order $n$. Then, the number of unknowns
is reduced by $n$ and the number of necessary relations is divided by $n$.
Since the cost of the linear algebra stage is $\lambda N^2$,
where $N$ is the size of the matrix and $\lambda$ is its average weight per row,
i.e.~the number of non-zero entries per row, we obtain the following
result.
\begin{fact} If $f$ and $g$ are two polynomials with
automorphisms $\sigma$ and $\tau$ of order $n$ verifying the hypothesis
of Theorem~\ref{th:galois}, then we have:
\begin{itemize}
    \item a speed-up by a factor $n$ in the sieve;
    \item a speed-up by a factor $n^2$ in the linear algebra stage.
\end{itemize}
\end{fact}

\bigskip

\noindent{\bf The particular case when $A_\sigma=A_\tau$.}
In Section~\ref{subsec: our polyselect}, we will present a method to select polynomials 
$f$ and $g$ with automorphisms $\sigma$ and $\tau$; it that $\sigma$ and
$\tau$ are expressed by the same rational fraction $A_\sigma=A_\tau$.
Moreover, the numerator and
denominator are constant or linear polynomials. A typical example is when
both polynomials are reciprocal and then $\sigma(\alpha)=1/\alpha$ and
$\tau(\beta)=1/\beta$.

Let $\phi\in\Z[x]$ be a polynomial yielding a relation. When we apply $\sigma$
and $\tau$ to the corresponding system of equations~\eqref{eq:doubly
smooth}, we get:
\begin{equation} \label{eq:conjugated}
    \left\{
\begin{array}{l}
\phi(A_\sigma(\alpha))\OO_f=\prod_{\gq\in\mF_f}(\gq^\sigma)^{\val_\gq\left(\phi(\alpha)\right)}\\
\phi(A_\tau(\tau))\OO_g=\prod_{\mathfrak{r}\in\mF_g}(\mathfrak{r}^\tau)^{\val_\mathfrak{r}\left(\phi(\beta)\right)},\\
\end{array} \right.  \end{equation}

Since $A_\sigma=A_\tau$ have a simple form, there is a chance that
$\phi\circ A_\sigma$ has a numerator that is again a polynomial of the
form that would be tested later. The relations being conjugates of each
others, the second one brings no new information and should not be
sieved. 

Again, we illustrate this on the example of reciprocal polynomials, where
$A_\sigma(x) = A_\tau(x) = 1/x$. For polynomials $\phi(x) = a-bx$ of
degree 1, the numerator of $\phi\circ A_\sigma$ is $b-ax$. Therefore, it
is interesting not to test the pair $(b,a)$ for smoothness if the pair
$(a,b)$ has already been tested.

If the sieve is implemented using the lattice sieve, e.g. in
CADO-NFS~\cite{CADO}, one can collect precisely these polynomials $\phi$
such that $\phi(\alpha)$ is divisible by one of the ideals $\gq$ in a
list given by the user. In this case, we make a list of ideals $\gq$
which contains exactly one ideal in each orbit
$\{\gq,\sigma(\gq),\ldots,\sigma^{n-1}(\gq)\}$. Hence, we do not collect
at the same time $\phi$ and the numerator of $\phi\circ A_\sigma$ except
if the decomposition of $\phi(\alpha)$ in ideals contains two ideals
$\gq$ and $\gq'$ which are in our list of ideals or conjugated to such an
ideal.

\section{Vanishing of the logarithms of units}
\label{sec:vanishing}
In this section, we are again in the case where we study the fields $K_f$
and $K_g$ independently. Therefore, we stick to the notations for the
$f$-side, but we keep in mind that this could be applied to $g$.
Furthermore, for easier reading, for this section we drop the subscript
$f$, for structures related to $f$:
$K=\Q(\alpha)$ is the number field of $f$, $U$ the unit group whose rank
is denoted by $r$, and $\rho$ is the reduction map to $\F_{p^n}$.

Also, some of the results of this section depend on the fact that $\ell$
is a factor of $p^n-1$ that is in the ``new'' part of the
multiplicative group: we will therefore always assume that 
$\ell$ is a prime factor of $\Phi_n(p)$.
The aim of this section is to give cases where the logarithms of some
or all fundamental units are zero, more precisely units $u$ for which
$\log\rho(u) \equiv 0 \bmod \ell$.

%%%%%%%%%% SSS
\subsection{Units in subfields}

The main case where we can observe units with zero virtual
logarithms is when the subfield fixed by an automorphism as in
Section~\ref{sec:galois} has some units.

\begin{theorem}\label{th:fixed subfield}
With the same notations as above, assume that  
$v_1,\ldots,v_r$ are units of $K$ which form a basis
modulo~$\ell$. Let $\sigma$ be an
automorphism of $K$ and assume that there exists an integer $A$ such that
$A\not\equiv 1\mod \ell$ and, for all $x\in K$ of norm coprime to $p$,
\begin{equation}\label{eq:log automorphisms}
\log\rho(\sigma(x))\equiv A \log\rho(x)\mod \ell.
\end{equation}
Let $K^{\langle \sigma\rangle}$ be the subfield fixed by
$\sigma$ and let $r'$ be its unit rank. Let $u'_1,\ldots,u'_{r'}$ be a
set of units of $K^{\langle \sigma\rangle}$ which form a basis
modulo~$\ell$. Then, $K$ admits a
basis $u_1,\ldots,u_r$ modulo~$\ell$ such that the discrete logarithms of
$\rho(u_1),\ldots,\rho(u_{r'})$ are zero modulo~$\ell$.
\end{theorem}

\begin{proof}
For any $x\in K^{\langle \sigma\rangle}$ we have $\sigma(x)=x$, so, when
$\rho$ is defined, we have
$\log(\rho(\sigma(x)))\equiv \log(\rho(x))\mod \ell$. Using
Equation~\eqref{eq:log automorphisms} we obtain that $\log(\rho(x))\equiv0\mod
\ell$ for all $x$ in $K^{\langle \sigma\rangle}$ of norm coprime to $p$. In particular,
for $1\leq i\leq r'$, we have $\log(\rho(u'_i))\equiv 0\mod \ell$.

One checks that $u'_1,\ldots,u'_{r'}$ are units in $K$. Since they form a basis
modulo~$\ell$, there is no non-trivial product of powers of
$u'_1,\ldots,u'_{r'}$ which is equal to an $\ell$th power. Then,
one can select $r-r'$ units among $v_1,\ldots,v_r$ to extend
$u'_1,\ldots,u'_{r'}$ to a basis modulo~$\ell$. 
\end{proof}

\begin{example}\label{ex:deg4}
Consider the family of CM polynomials
\begin{equation}
\begin{array}{l}
f=x^4+bx^3+ax^2+bx+1,\\
|a|<2,\hspace{1cm}|b|<2+a/2.
\end{array}
\end{equation}
There is always the automorphism,  $\forall T\in\Z[x],\sigma(T(x))=T(1/x)$ of order 2,
so that we have $A =p\equiv -1\mod \ell$ for use in the Theorem.
We claim that $r=r'=1$. Let us call $\alpha$ a complex root of $f$. Since
$\beta = \alpha+1/\alpha$ is not rational and fixed by $\sigma$, we have
$K^{\langle\sigma\rangle}=\Q(\alpha+1/\alpha)$. Since $\beta$
is a root of the equation
$P(Y) = Y^2+bY+(a-2)=0$, whose discriminant $b^2+4(2-a)$ is positive, $
K^{\langle\sigma\rangle} $ is real and we have $r'=1$. 
The roots of $f$ are roots of $x+1/x=y_1$ or $y_2$ for
$y_1=-b/2-\sqrt{b^2/4+(2-a)}$ and $y_2=
-b/2+\sqrt{b^2/4+(2-a)}$. Since $|b|<2+a/2$, $f$ has no real roots, so
$r=1$.

A second proof is as follows. Note that $f(X)$ factors over $\Q(\beta)$ as
$$(X^2-\beta X + 1) (X^2+(b+\beta)X+1).$$
We put $\varphi(X) = X^2-\beta X + 1$. Let $p$ be a prime for which
$P(Y)$ is reducible modulo $p$ and $\varphi$ is not.
The following picture shows the characteristic 0 picture, as well as
the one modulo $p$.

\medskip
\begin{center}
\setlength{\unitlength}{1mm}
\begin{picture}(60,25)
\put(0,20){\makebox(0,0){$K = \Q(\alpha) = \Q[X]/(f(X))$}}
\put(0,18){\line(0,-1){6}}
\put(0,10){\makebox(0,0){$K^{\langle\sigma\rangle} = \Q(\beta) = \Q[Y]/(P(Y))$}}
\put(0,8){\line(0,-1){5}}
\put(0,0){\makebox(0,0){$\Q$}}
\put(60,10){\makebox(0,0){$\GFpn{p}{2} = \GFq{p}[X]/(\overline{\varphi}(X))$}}
\put(60,8){\line(0,-1){5}}
\put(60,0){\makebox(0,0){$\GFq{p}$}}
\put(25,18){\vector(2,-1){15}}
\put(25,10){\vector(4,-1){30}}
\end{picture}
\end{center}

\bigskip
\noindent
Let $\ell\mid p+1$. If $\varepsilon_1$ is the fundamental unit of
$K^{\langle\sigma\rangle}$ (and also of $K$ by construction), we have
$\log\rho(\varepsilon_1) \equiv 0 \bmod \ell$.
\end{example}

%%%%%%%%%% SSS
\subsection{Extra vanishing due to $\F_\ell$-linear action}

In the previous section, we have just seen that with a careful choice of
the basis of units, some of the basis elements can have a zero virtual
logarithm. In general, there could be another choice for the basis that
give more zero logarithms. We call $\mathcal{R}_\text{opt}$ the maximum
number of units of $K$ in a basis modulo $\ell$ that can have zero
logarithm. With this notation, the result of Theorem~\ref{th:fixed
subfield} becomes $\mathcal{R}_\text{opt} \geq r'$.

The aim of this section it to prove a better lower bound for
$\mathcal{R}_\text{opt}$. By studying the $\F_\ell$-linear action of
$\sigma$ on the units, we will be able to choose a basis for which the 
number $\mathcal{R}$ of independent units with zero logarithm is (often)
larger than $r'$. Therefore, the notation $\mathcal{R}$ in this section
is a lower bound on the maximal number of units of $K$ in a basis modulo
$\ell$ that can have zero logarithm; and we always have
$\mathcal{R}_\text{opt} \geq \mathcal{R}$.
\medskip

For the unit group $U$ of $K$, consider the vector space $U/U^\ell$ over
$\F_\ell$. We assume that $\ell$ is large enough so that $K$ has no roots
of unity of order $\ell$; therefore the 
dimension of $U/U^\ell$ is equal to $r$.

We denote by $\overline{\sigma}$ the vector space homomorphism
$U/U^\ell\rightarrow U/U^\ell$, $\overline{\sigma}(u U^\ell)=\sigma(u)U^\ell$.
For simplicity, in the sequel, we drop the bar above $\overline{\sigma}$. 
Let $\mu_{\ell,\sigma}(x)$ be the minimal polynomial of $\sigma$; it
is a divisor of $x^n-1$, since $\sigma$ has order $n$. Note however that,
$\overline{\sigma}$ can have a smaller order than $\sigma$, as seen by
Example~\ref{ex:deg4} where $\sigma$ has order two but its restriction to
the unit group is the identity.

Since $\ell$ is a divisor of $\Phi_n(p)$, $\Phi_n(x)$ splits completely in
$\F_\ell$. Then, $x^n-1$ and $\mu_{\ell,\sigma}$ split completely in $\F_\ell$:
\begin{equation}
\mu_{\ell,\sigma}(x)=\prod_{i=1}^{\deg \mu_{\ell,\sigma}}(x-c_i),
\end{equation}
where $c_i$ are distinct elements of $\F_\ell$. We remark at this point
that as an endomorphism of $\F_\ell$-vector spaces, $\sigma$ is
diagonalizable.
For any eigenvalue $c\in\F_\ell$ of $\sigma$, we denote by $E_c$ the
eigenspace of $c$:
\begin{equation}
E_{c}=\left\{u\in U\mid \exists v\in U, \sigma(u)=u^{c}v^\ell\right\},
\end{equation}
and since the endomorphism is diagonalizable, the whole vector space can
be written as a direct sum of eigenspaces:
\begin{equation}
U/U^\ell=\prod_{i=1}^{\deg \mu_{\ell,\sigma}} E_{c_i}.
\end{equation}

The case covered by Theorem~\ref{th:fixed subfield} corresponds to the
units that are fixed by $\sigma$, i.e.~the units in the eigenspace $E_1$.
The following lemma generalizes the result to other eigenvalues.

\begin{lemma}
If $c\in \F_\ell$ is an eigenvalue distinct from $A$ (as defined in 
(\ref{eq:log automorphisms})), then, for all units $u$ such that the
class of $u$ in $U/U^{\ell}$
belongs to $E_{c}$, we have $\log(\rho(u))\equiv 0\mod \ell$. 
\end{lemma}
\begin{proof}
For such a unit $u$, we have 
\begin{equation*}
\log(\rho(\sigma(u)))\equiv c\log(\rho(u))\mod \ell.
\end{equation*}
By assumption on $A$, we have
\begin{equation*}
\log(\rho(\sigma(u)))\equiv A\log(\rho(u))\mod \ell.
\end{equation*}
We conclude that the logarithm of $\rho(u)$ is zero.
\end{proof}

\begin{corollary}\label{cor:R}
Using the notations above, we have $\mathcal{R} = r-\dim E_A$, where
$A$ is as in (\ref{eq:log automorphisms}).
\end{corollary}

At this stage, we get an expression that gives the units that have to be
considered during the NFS algorithm. But this expression depends on
$\ell$, whereas in many cases it will be inherited from global notions
and will be the same for any $\ell$ dividing $\Phi_n(p)$. Therefore we
consider the linear action of $\sigma$ on the group of non-torsion units.

Let
$U_\text{tor}$ be the torsion subgroup of $U$ and $\varepsilon_0$ a
generator of $U_\text{tor}$. Let
$\varepsilon_1,\varepsilon_2,\ldots,\varepsilon_r$ be a system of
fundamental units.
Let $M_\sigma$ be the matrix of the endomorphism $\overline{\sigma}$ on
$U/U_\text{tor}$, in basis $\varepsilon_1 U_\text{tor},\ldots,
\varepsilon_rU_\text{tor}$. Then
$M_\sigma$ belongs to $\GL_r(\Z)$. Since $M_\sigma$ cancels the monic
polynomial $x^n-1$, $M_\sigma$ admits a minimal polynomial
$\mu_{\Z,\sigma}$ with integer coefficients. Note that $\mu_{\Z,\sigma}$
does not depend on the
system of fundamental units used. The following lemma shows that finding
roots modulo $\ell$ of $\mu_{\Z,\sigma}$ gives local information about
the vanishing of the logarithms modulo $\ell$.
 
\begin{lemma}\label{lem:non emptyness}
For any root $c\in\F_\ell$ of $\mu_{\Z,\sigma}(x)$ the dimension of the
eigenspace $\dim(E_c)$ in $U/U^\ell$ is at least $1$.

\end{lemma}
\begin{proof}
Since $M_\sigma$ has integer coefficients, its characteristic polynomial
$\chi_{M_\sigma}$ is monic with integer coefficients. We deduce that $\mu_{\Z,\sigma}$ is
monic with integer coefficients. We claim that, for all primes $\ell$,
\begin{equation}\label{eq:mu}
\mu_{\Z,\sigma}=\mu_{\ell,\sigma}.
\end{equation}
On the one hand, $\mu_{\Z,\sigma}$ has the same irreducible factors over
$\Q$ as the characteristic polynomial $\chi_{M_\sigma}$ of $M_\sigma$. Since $\sigma$
cancels $x^n-1$, they occur with multiplicity one in $\mu_{\Z,\sigma}$. 
Hence, $\mu_{\Z,\sigma}$ is the product of irreducible factors of
$\chi_{M_\sigma}$, taken with multiplicity one.

On the other hand, $\mu_{\ell,\sigma}$ has the same irreducible factors as the
characteristic polynomial of $\overline{\sigma}$, which is the reduction of
$\chi_{M_\sigma}$ modulo $\ell$. Since, $\overline{\sigma}$ cancels
$x^n-1$, $\mu_{\ell,\sigma}$ is product of the irreducible factors of
$\chi_{\Z,\sigma}$ modulo $\ell$, taken with multiplicity one. We obtain
equation~\eqref{eq:mu}.   

Finally, it is a classic property of minimal polynomial that all its roots
have nonzero eigenspaces.
\end{proof}

We already mentioned the link between the eigenspace of $1$ and
Theorem~\ref{th:fixed subfield}. We now make this more precise:

\begin{lemma}\label{lem:dim E_1}
Using the previous notations, we have $\dim E_1=r'$, except for a finite
set of primes~$\ell$.
\end{lemma}
\begin{proof}
Consider a system of fundamental units. By
Theorem~\ref{th:fixed subfield} there exists a basis $(u_i)$, $1 \leq
i\leq r$, of $U/U_\text{tors}$ such that the first $r'$ elements are
fixed by $\sigma$ and, no unit in the subgroup $V$ generated by $u_i$,
$r'+1\leq i\leq r$, is fixed by $\sigma$. After block-diagonalization,
we can assume that $V/V_\text{tors}$ is stable by $\sigma$ and we
let $M_\sigma$ be the matrix
of $\sigma$ on $V/V_\text{tors}$. The determinant of
$(M_\sigma-\text{id})$ is an integer $D$. If $\ell$ is prime to
$D$, the discriminant of $\sigma$ on $V/V^\ell$ is
non-zero. Hence, $\dim E_1\leq r'$, which completes the proof.
\end{proof}

As a first application, we study the case of cyclic extensions of
prime degree.
\begin{proposition}\label{prop:cycl}
Let $n$ be an odd prime and $K/\Q$ a cyclic Galois
extension of degree $n$. Let $p$ and $\ell$ be two primes such that
$\Phi_n(p)$ is divisible by $\ell$. Let $\rho$ be a ring morphism which
sends any element $x$ of $K$ with $\nu_p(x) \geq 0$ into the
field $\F_{p^n}$. Let $\sigma$ be an automorphism of $K$
of order $n$ for which there exists a constant $\kappa$ such that, for
all $x\in K$ of positive $p$-valuation, $\rho(\sigma(x))=p^\kappa \rho(x)
$. Then we have $\mathcal{R}=n-2$.
\end{proposition}
\begin{proof}
We want to compute $\mu_{\Z,\sigma}$. Since $\sigma$ has order $n$,
$\mu_{\Z,\sigma}$ is a
divisor of $(x^n-1)=(x-1)\Phi_n(x)$. 
By Lemma~\ref{lem:dim E_1}, $\dim \ker (\sigma-\text{id})=r'$, the
unit rank of the subfield fixed by $\sigma$. In our case, this
subfield is $\Q$, so $r'=0$. Then, we have $\mu_{\Z,\sigma}=\Phi_n(x)$. 

Let $f$ be a defining polynomial of $K$. Since $f$ has odd degree, it has
at least a real root~$\alpha$. Since $K$ is Galois, $K=\Q(\alpha)$, so all
roots of $K$ are real, hence its unit rank is $n-1$. By Lemma~\ref{lem:non
emptyness}, since $\deg(\Phi_n)=n-1=\dim U/U^\ell$, all the eigenspaces of
roots of $\Phi_n$ have dimension one. Using Corollary~\ref{cor:R}, we have
$\mathcal{R}=n-2$.

\end{proof}

%%%%%%%%%% SSS
\subsection{Fields of small degree}

We are now in position to list possible cases for fields of small degree.
As a warm-up, we start with the case of degree 2. The imaginary case is
of course trivial, since the unit rank is 0. For the real quadratic case,
the unit rank $r$ is 1, and one could wonder whether there are cases when
we can tell in advance that the virtual logarithm of the unit is 0
modulo $\ell$ with our method.

In fact, this does not occur. Indeed, the automorphism $\sigma$ is of
order $2$ and therefore the unit rank $r'$ of the subfield is 0. By
Lemma~\ref{lem:dim E_1} the dimension of the eigenspace $E_1$ is
therefore 0 as well. This is no surprise: the fundamental unit is not
defined over $\Q$, so it is not fixed by $\sigma$. The next step is to
study $\mu_{\Z,\sigma}$. Since $\sigma$ as order $2$, $\mu_{\Z,\sigma}$
divides $x^2-1$, and we have just seen that 1 is not an eigenvalue.
Therefore we deduce that $\mu_{\Z,\sigma} = x+1$. Hence the vector space
$U/U^\ell$ is reduced to the eigenspace $E_{-1}$. Since $-1$ is precisely
the value $A$ as in (\ref{eq:log automorphisms}), we can not conclude.
\medskip

The cases of degree 3 and 5 are covered by Proposition~\ref{prop:cycl}.
In Table~\ref{tab:cases of automorphisms}, we list the cases for degree 4
and 6. In all cases, a classification according to the signatures of the
field and of the fixed subfield is enough to conclude about the value of
$\mathcal{R}$. 

\begin{theorem}\label{thm:fdal}
The values of $\mathcal{R}$ for $K/\Q$ of degree 4
or 6 having non-trivial automorphisms are as given in
Table~\ref{tab:cases of automorphisms}.
\end{theorem}
\begin{table}
\begin{tabular}{|c|c|l|l|c|c|c|}
\hline
$\deg(K)$ & $\ord(\sigma)$ & sign($K$), sign($K^{\langle \sigma\rangle}$) &\qquad\qquad  $\mu_{\Z,\sigma}$      & $\mathcal{R}$ & $r$  & $r-\mathcal{R}$      \\

\hline
\hline
4        &     2          & (4,0), (2,0)         &   $(x-1)(x+1)$             & 1    & 3   & 2           \\
         &                & (2,1), (2,0)         &   $(x-1)(x+1)$             & 1    & 2   & 1           \\
         &                & (0,2), (0,1)         &   $x+1$                    & 0    & 1   & 1           \\
         &                & (0,2), (2,0)         &   $x-1$                    & 1    & 1   & 0           \\
\cline{2-7}
         &     4          & (4,0), -             &   $(x+1)(x^2+1)$           & 2    & 3   & 1           \\
         &                & (0,2), -             &   $x+1$                    & 1    & 1   & 0           \\ 
\hline

   6     &     2          &  (0,3), (1,1)        &  $(x-1)(x+1)$             &  1 & 2   & 1           \\
         &                &  (0,3), (3,0)        &  $x-1$                    &  2 & 2   & 0           \\
\cline{2-7}
         &     3          & (6,0), (2,0)         &   $(x-1)(x^2+x+1)$         & 3    & 5   & 2           \\
         &                & (0,3), (0,1)         &   $x^2+x+1$              & 1    & 2   & 1           \\
\cline{2-7}
         &     6          & (6,0), -             &  $(x+1)(x^2+x+1)(x^2-x+1)$                     & 4 & 5   & 1           \\
         &                & (0,3), -             &  $x^2+x+1$                     & 1 & 2   & 1          \\
\hline
\end{tabular}
\caption{Table of values of $\mathcal{R}$ for fields of degree 4 and 6.}
\label{tab:cases of automorphisms}
\end{table}

\begin{proof}
Let us consider the various cases of Tab.~\ref{tab:cases of
automorphisms}. In each case, we use a strategy of proof that is not so
different from the real quadratic case that we mentioned in
introduction.  In
order to determine the minimal polynomial $\mu_{\Z,\sigma}$, we consider the
factors of $x^n-1$ in $\Z[x]$ and we use the fact that $\deg \mu_{\ell,\sigma}$ is
at most $\dim(U/U^\ell)=r$.
\goodbreak

\underline{Case $\deg(K)=4$ and $\ord(\sigma)=2$} 
\begin{itemize}
\item Case when sign($K$)=(4,0) and sign($K^{\langle \sigma
\rangle}$)=(2,0). Then, $r=3$ and $r'=1$. Further, $x-1$
divides $\mu_{\Z,\sigma}$ with multiplicity one. Since $\sigma$
cancels $x^2-1$, the minimal polynomial is
$\mu_{\Z,\sigma}=x^2-1$. Hence, we have $\dim E_{-1}=2$. Since
$A=-1$, we obtain $\mathcal{R}=r'=1$.
\item Case when sign($K$)=(2,1) and sign($K^{\langle \sigma
\rangle}$)=(2,0). Note
first that (2,0) is the unique possibility for sign($K^{\langle \sigma
\rangle}$). Indeed, if $\alpha$ is a real root of a defining polynomial of $K$,
then $\Q\left(\alpha+\sigma(\alpha)\right)$ is fixed by $\sigma$ and has degree
two, so it is $K^{\langle \sigma \rangle}$. Since this quadratic field is real,
its signature is (2,0). As above we have that $\dim E_1\geq 1$ and therefore
$\mu_{\Z,\sigma}=(x-1)(x+1)$, implying that $\mathcal{R}=1$.
\item Case when sign($K$)=(0,2) and sign($K^{\langle \sigma \rangle}$)=(0,1).
Here $r=1$ and $r'=0$. For sufficiently large values of $\ell$, by
Lemma~\ref{lem:dim E_1}, the space $E_1$ of units fixed by $\sigma$ has
dimension $r'=0$. Since the minimal polynomial divides $x^2-1$, we
have 
$\mu_{\Z,\sigma}=x+1$. We deduce that $\mathcal{R}=0$.
\item Case when sign($K$)=(0,2) and sign($K^{\langle \sigma \rangle}$)=(2,0).
Here we have $r=r'=1$. Then any unit is fixed by $\sigma$ and
$E_1=U/U^\ell$, so $\mu_{\Z,\sigma}=x-1$ and $\mathcal{R}=1$.
We remark that the fields of Example~\ref{ex:deg4} falls in this
category.
\end{itemize}
\goodbreak

\underline{Case $\deg(K)=4$ and $\ord(\sigma)=4$}
 Note that $K$ is either totally
real or its defining polynomial has no real root.
Hence we have two cases:
\begin{itemize}    
\item Case when sign($K$)=(4,0).
Here we can apply to $\tau=\sigma^2$ the results on the case of degree four
polynomials with automorphisms of order two. Hence we have
$\dim\ker(\sigma^2-1)=\dim\ker
(\tau-1)=1$ and $\dim(\sigma^2+1)=\dim(\tau+1)=2$. 
The fixed field has
degree $1$, so $r'=0$. Then, the minimal polynomial is
$\mu_{\Z,\sigma}=(x+1)(x^2+1)$, and we have $\mathcal{R}=2$. 
\item Case when sign($K$)=(0,2). Here the unit rank of $K$ is $r=1$, so the
minimal polynomial is linear. Since, $r'=0$, we have $\dim E_1=0$, so
$\mu_{\Z,\sigma}=x+1$. Since $A$ is of order $4$ it is not $-1$; hence, we
obtain $\mathcal{R}=0$. Note that here the group automorphism
$\overline{\sigma}$ equals $-1$, so it has a smaller order that the field
automorphism $\sigma$. 
\end{itemize}
\goodbreak

\underline{Case $\deg(K)=6$ and $\ord(\sigma)=2$}

Here the signature of $K$ can be $(6,0)$, $(4,1)$, $(2,2)$ and
$(0,3)$. We only deal with the case $(0,3)$ in the present version of
our work.

The unit rank of $K$ is $r=2$
and the minimal polynomial is a factor of $x^2-1$. The value of
$\mathcal{R}$ is determined by the signature of the subfield fixed by
$\sigma$, which is cubic and can have signature (3,0) or (1,1).
\begin{itemize}
\item Case when sign($K^{\langle \sigma \rangle}$)=(1,1). The unit rank of the
fixed subfield is
$r'=1$ and $\mathcal{R}=1$.
\item Case when sign($K^{\langle \sigma \rangle}$)=(3,0). Here we have
$r'=2$, so $\dim E_1=r$ and $\mu_{\Z,\sigma}=x-1$. This shows
that $\mathcal{R}=2.$
\end{itemize}
\goodbreak

\underline{Case $\deg (K)=6$ and $\ord(\sigma)=3$}

Note that, the
signature ($r_\R$,$r_\C$) of $K$ satisfies $r_\R\equiv 0\mod 3$. Indeed, if a
defining polynomial of $K$ has a real root $\alpha$, the roots $\sigma(\alpha)$
and $\sigma^2(\alpha)$ are also real. The two values for the signature
are (6,0) and (0,3).
\begin{itemize}
\item Case when sign($K$)=(6,0). Since $K$ is real, $K^{\langle \sigma
\rangle}$
is also real, so $r'=1$. As in the previous cases, the polynomial
$\mu_{\Z,\sigma}$ is a factor of $(x-1)(x^2+x+1)$. Since, $\dim
E_1=r'=1$ is
neither $0$ nor $r$, we have $\mu_{\Z,\sigma}=(x-1)(x^2+x+1)$. Since, the
characteristic polynomial over $\Q$, of $\sigma$ restricted to 
$V=\text{ker}(\sigma^2+\sigma+1)$ has the same irreducible factors, we
have
$\chi_\sigma|_V=(x^2+x+1)^2$. Suppose ab absurdo that, for a root $c$ of
$\mu_{\Z,\sigma}$ modulo $\ell$, we have $\dim E_c\geq 3$. Then $\chi_{\sigma}$ modulo $\ell$ is
divisible by $(x-c)^3$. It is impossible because it has two roots of
multiplicity at least two, so, for the two roots of $x^2+x+1$
modulo $\ell$ we have $\dim E_c=2$. It implies that $\mathcal{R}=3$. 
\item Case when sign(K)=(0,3). The unit rank of $K$ is $r=2$, so the minimal
polynomial is $x-1$ or $x^2+x+1$. The fixed subgroup has degree $2$, so we
cannot have $r'=2$. This shows that $\mu_{\Z,\sigma}=x^2+x+1$. Then,
$r'=0$ and $\mathcal{R}=1$.
\end{itemize}
\goodbreak

\underline{Case $\deg(K)=6$ and $\ord(\sigma)=6$}

As in the case of cyclic
quartic Galois extensions, either $K$ is real or has no real roots. 
\begin{itemize}
\item Case when sign($K$)=(6,0). The unit rank of $K$ is $5$, so the minimal
polynomial is equal to a factor of $(x-1)(x+1)(x^2+x+1)(x^2-x+1)$,
having degree less than or equal to $5$. Since the fixed subgroup is $\Q$,
we have $r'=0$, so $\dim E_1=0$. The fixed subfield of $\sigma^3$ is a
cubic cyclic Galois extension, so its unit rank is two. Hence
$\dim\text{ker}(\sigma^3-1 )=2$. Since $\dim E_1=0$, $x^2+x+1$ divides
$\mu_{\Z,\sigma}$. We also deduce that $\dim \text{ker}(\sigma^3+1)=3$. The
subfield fixed by $\sigma^2$ has degree two, so its unit rank is at
most one. It
implies that $\dim(\ker(\sigma+1))\neq 3$, so $x^2-x+1$ divides
$\mu_{\Z,\sigma}$. Since the dimension of its kernel is even, $\dim E_{-1}\neq
0$, so $(x+1)$ divides $\mu_{\Z,\sigma}$. We conclude that
$\mu_{\Z,\sigma}=(x+1)(x^2+x+1)(x^2-x+1)$ and $\mathcal{R}=4$.
\item Case when sign($K$)=(0,3). The subfield fixed by $\sigma^3$ is a cyclic
cubic extension of $\Q$, so its unit rank is $2$. This means that $\dim
\text{ker}(\sigma^3-1)=2$, so the minimal polynomial divides $x^3-1$. The fixed
subgroup of $\sigma$ is $\Q$, so $\dim E_1=0$ and $\mu_{\Z,\sigma}=x^2+x+1$. We
obtain that $\mathcal{R}=1$. 
\end{itemize}

\end{proof}

%%%%%%%%% SS 
\subsection{Effective computations}

Theorem \ref{thm:fdal} tells us that we do not need to consider the
logarithms of all units in many cases. If we have a system of units which
form a basis modulo $\ell$, we can make the theorem effective by solving
linear systems. This is a less stronger condition than computing a system
of fundamental units. One can investigate the use of Schirokauer maps to 
avoid any requirement of effective computations of units. Let us see a
series of examples which illustrate Theorem \ref{thm:fdal}.

%%%%%%%%%%%%%%% SSS 
\subsubsection{Minkowski units}
\label{sec:minkowski}

A {\em Minkowski unit} for $K$, if it exists, is a unit $\varepsilon$
such that a subset of the conjugates of $\varepsilon$ forms a system
of fundamental units. Some results on the classification of such
fields exist, we will come back to them in the final version of this
work.

As an example, when $K$ is totally cyclic of degree 3, it is real and
there exists always a Minkowski unit as shown by Hasse
\cite[p. 20]{Hasse48}. In that case, using the proof in \ref{th:galois}, we
see that
$$\log\rho(\varepsilon^{\sigma}) \equiv p^{\kappa}
\log\rho(\varepsilon) \bmod \ell,$$
so that we need to find $\log\rho(\varepsilon) \bmod \ell$ only. It matches
Table~\ref{tab:cases of automorphisms}, where we read that only
$r-\mathcal{R}=3-2=1$ (well chosen) Schirokauer map is required.

%%%%%%%%%%%%%%% SSS 
\subsubsection{The degree 4 cases}
\label{sec:exdeg4}

When the signature is $(4, 0)$ and the Galois group is $C_4$, we can
precise the structure of $U_K$, see \cite{Hasse48,Gras79}.
The first case is when $K$ admits a Minkowski unit, that is
$\varepsilon$ such that $U_K = \langle -1, \varepsilon,
\varepsilon^{\sigma}, \varepsilon^{\sigma^2}\rangle$. And we use the
same reasoning as in Section \ref{sec:minkowski} to reduce the number
of logarithms needed to $1$.

In the second case, $U_K = \langle -1, \epsilon_1, \epsilon_\chi,
\epsilon_\chi^{\sigma}\rangle$, where $\epsilon_1$ is the fundamental
unit of the quadratic subfield and $\epsilon_\chi$ is a generator
of the group of relative units, that is $\eta \in U_K$ such that
$\Norm_{K/K_2}(\eta) = \pm 1$. We gain two logarithms since we can use
the Galois action for $\log\rho(\epsilon_\chi^{\sigma})$, and we know
$\log\rho(\epsilon_1)$.

Note that Table~\ref{tab:cases of automorphisms} predicts that the two
cases above, with or without relative units, lead to the same number of
Schirokauer maps: $r-\mathcal{R}=1$.

\medskip
For signature $(0, 2)$, the rank of $K$ is $1$, and the
fundamental unit is that from the real quadratic subfield, so we don't
need any logarithm at all. To be more precise, let us detail the case
of our favorite example: $f(X) = X^4+1$ which defines the $8$-th roots
of unity, say $K = \Q(\zeta_8)$. The Galois group of $f$ is $V_4$ and
$K$ has two automorphisms
$\sigma_1: x \mapsto -x$, $\sigma_2: x \mapsto 1/x$. We compute
that
$$K^{\langle\sigma_1\rangle} = \Q(i),
K^{\langle\sigma_2\rangle} = \Q(\sqrt{2}),
K^{\langle\sigma_1\sigma_2\rangle} = \Q(\sqrt{-2}).$$
The corresponding factorizations of $f(X)$ are
$$f(X) = (X^2+i) (X^2-i),$$
$$f(X) = (X^2 - \sqrt{2} X + 1) (X^2 + \sqrt{2} X + 1),$$
$$f(X) = (X^2 - \sqrt{-2} X - 1) (X^2 + \sqrt{-2} X - 1).$$
Since $f$ has signature $(0, 2)$, we have $U_K = \langle
\zeta_8\rangle \times \langle\varepsilon\rangle$, where $\varepsilon$
comes from $K^{\langle\sigma_2\rangle}$, the only real quadratic subfield.
By Theorem \ref{thm:fdal}, we do not need any logarithm of units for
use in NFS. 

\section{Reducing the number of Schirokauer maps}
\label{sec:units}

In this section, we use the preceding Section to conclude that we can
reduce the number of Schirokauer maps needed in NFS-DL.

We use the notations of Section \ref{sec:vanishing}. A system
of $r$ units is a basis modulo~$\ell$ if its image in $U/U^\ell$ is a
basis. Let $p$ be a prime and $n$ an integer such that the reduction
of $g$ modulo~$p$ has an irreducible factor of degree~$n$. Let $\ell$
be a prime factor of $p^n-1$, coprime to $p-1$. In order to reduce the
number of Schirokauer maps associated to $g$, we follow the steps below:
\begin{enumerate}
\item We find a system of $r$ elements in $U$ which form a basis
modulo~$\ell$.
\item We compute an integer $\mathcal{R}\leq r$, as large as possible, and a
system of $r$ elements $u_1,\ldots,u_r$ in $U$ which form
a basis modulo~$\ell$, such that the discrete logarithms of
$\rho(u_1),\ldots,\rho(u_{\mathcal{R}})$ are zero modulo~$\ell$. 
\item Using any set of $r$ Schirokauer maps and the system of
fundamental units above, we compute a set of Schirokauer maps
$\lambda_1,\ldots,\lambda_r$ such that
the NFS algorithm can be run using only the last $r-\mathcal{R}$ Schirokauer maps
$\lambda_{\mathcal{R}+1},\ldots,\lambda_{r}$. 
\end{enumerate}

We do not discuss the first point here. Point
(2) was studied in Section \ref{sec:vanishing}. Point $(3)$ is solved by
the corollary of the following theorem.

\begin{theorem}
Let $\lambda_1,\ldots,\lambda_r$ be a set of Schirokauer maps and let 
$u_1,\ldots,u_r$ a system of effectively computed units in $U$, whose image in
$U/U^\ell$ form a basis. Then there exits a set of effectively computable
Schirokauer maps $\lambda'_1,\ldots,\lambda'_r$ such that, for $1\leq i,j\leq r$, 
\begin{equation}\label{eq:dual of lambda}
\lambda'_i(u_j)=\left\{\begin{array}{l}
1\qquad \text{ if }i=j,\\
0\qquad \text{ otherwise.} 				      
  \end{array}\right.
\end{equation}
\end{theorem}
\begin{proof}
Let $L=(l_{i,j})$ be the $r\times r$ matrix of entries $l_{i,j}=\lambda_i(u_j)$.
Let $C=(c_{i,j})$ be the inverse of $L$. Then, we put
\begin{equation}
\lambda'_i=\sum_{n=1}^r  c_{i,n} \lambda_n .
\end{equation}
We have $\lambda'_i(u_j)=(CL)_{i,j}$, so the maps $\lambda'_i$ verify the condition in
Equation~\eqref{eq:dual of lambda}.
\end{proof}

\begin{corollary}
Let $u_1,\ldots,u_r$ be a set of units, effectively computed, which form a basis
modulo~$\ell$. Assume that for some integer $\mathcal{R}$, $1\leq \mathcal{R}\leq r$, the first
$\mathcal{R}$ units $u_1,\ldots,u_{\mathcal{R}}$ are such that $\log\left(\rho(u_i)\right)\equiv
0\mod \ell$. Then, there exists a set of $r$ effectively computable Schirokauer
maps $\lambda'_1,\ldots,\lambda'_r$ such that NFS can be run with the last $r-\mathcal{R}$ maps instead of the complete set of $r$ maps. 
\end{corollary}

\begin{proof}
Using any set of Schirokauer maps, we compute $\lambda'_1,\ldots,\lambda'_r$
such that Equation~\eqref{eq:dual of lambda} holds.

By Equation~\eqref{eq:implicit} in Section~\ref{ssec:virtual logarithms}, when
running NFS with the maps $\lambda'_1,\ldots,\lambda'_r$, the linear
algebra stage computes the virtual logarithms of the ideals in the factor base
together with $r$ constants $\chi_i$, $1\leq i\leq r$ such that
\begin{equation}\label{eq:sum lambda}
  \log(\rho(\gamma))\equiv \sum_{\gq\in \mathcal{F}}\log \gq
\val_{\gq}(\gamma) + \sum_{i=1}^r \chi_i \lambda_i(\gamma) \mod \ell.
\end{equation}
For $1\leq i\leq r$, when injecting $\gamma=u_i$ in Equation~\eqref{eq:sum lambda} we
obtain that $\chi_i=\log(\rho(u_i))$. For $1\leq i\leq \mathcal{R}$ we have
$\log(\rho(u_i)\equiv 0\mod \ell$, and therefore $\chi_i\equiv0\mod
\ell$. Hence, Equation~\eqref{eq:sum lambda} can be
rewritten with $r-\mathcal{R}$ Schirokauer maps:
\begin{equation}\label{eq:sum lambda 2}
  \log(\rho(\gamma))\equiv \sum_{\gq\in \mathcal{F}}\log \gq
\val_{\gq}(\gamma) + \sum_{i=\mathcal{R}+1}^{r} \chi_i \lambda_i(\gamma)
\mod \ell.
\end{equation}

\end{proof}

\begin{example} (continued)
The corollary above states that the polynomials in the family
described in Example \ref{ex:deg4} do not require any
Schirokauer map. Moreover, note that if $f_1$ and $f_2$ are two
polynomials in this family and $\mu_1,\mu_2$ are two positive
rationals such that $\mu_1+\mu_2=1$, then $\mu_1f_1+\mu_2f_2$ also
belongs to this family.
\end{example}

A more important example is that of cubic polynomials with an automorphism
of order three. Then, we can effectively compute a linear combination
$\Lambda_{1,2}$ of
any two Schirokauer maps $\Lambda_1$ and $\Lambda_2$ so that NFS can be run
with $\Lambda_{1,2}$ as unique Schirokauer map.

\section{Two new methods of polynomial selection}
\label{sec:polyselect}
In this section, we propose two new methods to select the polynomials
$f$ and $g$, in the case of finite fields that are low degree extensions
of prime fields. The first one is an extension to non-prime fields of the method
used by Joux and Lercier~\cite{JoLe03} for prime fields. The second one, which 
relies heavily 
on rational reconstruction, insists on having coefficients of size
$O(\sqrt{p})$ for $g$. For both methods, $f$ has very small coefficients, 
of size $O(\log p)$.

\subsection{The state-of-art methods of polynomial selection}
\label{subsec:JLSV06 polyselect}

Joux, Lercier, Smart and Vercauteren~\cite{JLSV06} introduced two methods of
polynomial selection, one which is the only option for medium characteristic
finite fields and one which is the only known for the non-prime large
characteristic fields. 

\subsubsection{The first method of JLSV}\label{sssec:JLSV1}
Described in \cite[\S 2.3]{JLSV06}, this method is best adapted to the medium
characteristic case. It produces two polynomials $f,g$ of same degree $n$, which
have coefficients of size $\sqrt{p}$ each. 

One starts by selecting a polynomial $f$ of the form $f = f_v + a f_u$ with a
parameter  $a$ to be chosen. Then one computes a rational reconstruction
$(u,v)$ of $a$ modulo $p$ and one defines $ g = v f_v + u f_u$. Note that, by
construction, we have $f = v \cdot g \bmod p$. Also note that both polynomials
have coefficients of size $\sqrt{p}$.

\begin{example}
Take $p = 1000001447$
and $a = 44723 \geqslant \lceil \sqrt{p} \rceil$. One has 
$ f = x^4 - 44723 x^3 - 6 x^2 + 44723 x + 1$ and 
$g = 22360 x^4 - 4833 x^3 - 134160 x^2 + 4833 x + 22360$ with 
$u/v = 4833 / 22360 $ a rational reconstruction of $a$ modulo $p$.
\end{example}

 The norm product is 
$N_f N_g = 
E^{2n} p = E^{2n} Q^{1/n}$.

If one wants to use automorphisms as in Section~\ref{sec:polyselect}, then one
chooses $f$ in a family of polynomials which admit automorphisms. For example when $n=4$, one can  
take $f$ in the family presented in Tab.~\ref{tab:foster}, formed of 
degree 4 polynomials with cyclic Galois group of
order four, having an explicit automorphism:
 $f = (x^4 - 6 x^2 + 1) + a (x^3 - x) = 
f_v + a f_u$. Note that the second polynomial $g$ belongs to the same family and
has the same automorphisms. 

\subsubsection{The second method of JLSV}\label{sssec:JLSV2}

The second method is described in \cite[\S 3.2]{JLSV06}. It starts by 
computing $g$ of degree $n$ then it computes $f$ of degree $\deg f 
\geqslant n$.
First one selects
$g_0$ of degree $n$ and small coefficients. Then one chooses an integer 
$W \sim p^{1/(d +1)}$, but slightly larger, and set $g = g_0(x + W)$. The  smallest degree 
coefficient of $g$ has size $W^n$. We need to take into account the 
skewness of the coefficients. 
 The polynomial $f$ is computed by reducing the lattice of polynomials of degree
at most $d$, divisible by $g$ modulo $p$. We do this by defining the matrix $M$ in 
Sec.~\ref{subsec: Joux Lercier polyselect}, 
eq.~\eqref{eq:LLL generalization}, with $\varphi = g$. We obtain a polynomial $f$ with coefficients of size
$p^{n / (d +1)} = Q^{1/(d + 1)}$.

\begin{example}
Consider again the case of $p=1125899906842783$ and $n=4$. We take $g_0$ a polynomial of degree four and small coefficients, for example 
$g_0 = x^4 - x^3 - 6x^2 + x + 1$.  We can have $\deg(f)=d$ for any value of $d\geq n$, we take $d=7$ for the
example. We use $W = 77\geqslant p^{1/(d+1)} $, where we emphasize that we do
not use $Q^{1/(d+1)}$, and we set
\[g=g_0(x+W)=x^4 + 307x^3 + 35337x^2 + 1807422x + 34661012.\]
 We construct the lattice of polynomials of degree at most $7$ which are
divisible by $g$ modulo $p$. We obtain 
\[\begin{array}{c}f=12132118x^7 + 11818855x^6 + 2154686x^5 \qquad\qquad\qquad\\
\qquad\qquad\qquad- 7076039x^4 + 7796873x^3 + 7685308x^2 + 4129660x - 14538916.\end{array}\]
 Note that $f$ and $g$ have coefficients of size $Q^{1/8}$. 
\end{example}

For comparison, we compute the norms' product: $E^{d+n}Q^{2/(d+1)}$. However,
one might obtain a better norms product using the skewness notion introduced by
Murphy~\cite{Mur99}. Without entering into details, we use as a lower bound for
the norms product the quantity $E^{d+n}Q^{3/2(d+1)}$. Indeed, the coefficients of
$f$ have size $Q^{1/(d+1)}$ and the coefficients of $g$ have size $Q^{1/(d+1)}$,
which cannot be improved more than $Q^{1/2(d+1)}$ using skewness. This bound is optimistic,
but even so the new methods will offer better performances.  

\subsection{The generalized Joux-Lercier method}
\label{subsec: Joux Lercier polyselect}

In the context of prime fields,
Joux and Lercier proposed a method in \cite{JoLe03} to select polynomials
using lattice reduction. They start with a polynomial $f$ of degree $d+1$
with small coefficients, such that $f$ admits a root $m$ modulo $p$.
Then, a matrix $M$ is constructed with rows that generate the lattice of
polynomials of degree at most $d$ with integer coefficients, that also
admits $m$ as a root modulo $p$. We denote by $\LLL(M)$ the matrix 
obtained by applying the LLL algorithm to the rows of $M$:

\begin{footnotesize} 
\begin{equation} \label{eq:LLL}M = \left[
\begin{array}{cccc} 
p      &  0    & \cdots &  0     \\
-m     &  1    &   0    &  0     \\
\vdots &\ddots & \ddots &  0     \\
-m^d   &\cdots & 0      &  1     \\
\end{array}
\right],
 ~ ~ \LLL(M) =  \left[\begin{array}{c@{~~}c@{~~~}c@{~~~}
c} 
g_0    & g_1    & \cdots &  g_d    \\
*      & *      & \cdots & *       \\
\vdots & \vdots & \ddots & \vdots  \\
*      & *      & \cdots & *       \\
 \end{array}\right]~.  
\end{equation} 
\end{footnotesize}

The first row gives a polynomial $g = g_0 + g_1 x + \ldots + g_d x^d$,
that has a common root $m$ with $f$, and the pair of polynomials $(f,g)$
can be used for computing discrete logarithm in $\F_p$ with the NFS
algorithm.
\bigskip

In order to tackle discrete logarithms in $\F_{p^n}$,
we generalize this to polynomials $(f,g)$ that share an irreducible common
factor $\varphi$ of degree $n$ modulo $p$. Let $d'$ be an integer parameter
that we choose below. We select $f$ an irreducible
polynomial in $\ZZ[x]$ of degree $d^{'}+1 \geqslant n$, with small
coefficients, good sieving properties and an irreducible degree $n$
factor $\varphi= \sum_{i=0}^{n} \varphi_i x^i$ modulo $p$, that we force
to be monic. We
define a $(d'+1)\times(d'+1)$ matrix $M$ whose rows generate the lattice
of polynomials of degree at most $d'+1$ for which $\varphi$ is also a
factor modulo $p$. Then, running the LLL algorithm on this matrix gives
a matrix whose rows are generators with smaller coefficients. A possible
choice for the matrix $M$ is as follows, where the missing coefficients
understand to be zero.

\begin{footnotesize}
\begin{equation} \label{eq:LLL generalization}
M = \left[\begin{array}{cccccc} p            &
&           & &       &              \\ & \ddots      &           &
&       & \\ &             & p         &                &       &
\\ \varphi_{0} & \varphi_{1}& \cdots    & \varphi_{n}    &       & \\
& \ddots     &\ddots      &                &\ddots &              \\ &
&\varphi_{0}& \varphi_{1}    &\cdots & \varphi_{n} \\
\end{array}\right] \begin{array}{l} \left \rbrace \begin{array}{l} \\ \deg
\varphi = n \mbox{ rows} \\ \\ \end{array}\right. \\ \left \rbrace
\begin{array}{l} \\ \deg g + 1 - \deg \varphi \\ = d'+1-n \mbox{ rows}
\\ \end{array}\right.
\end{array}
\LLL(M) =  \left[\begin{array}{c c c c}
        g_0 & g_1 & \cdots & g_{d'} \\
	&    &        &        \\
        &    &        &        \\
        &     \multicolumn{2}{c}{*}    & \\
	&    &        &         \\
        &    &        &         \\
\end{array}\right]~.
\end{equation}
\end{footnotesize}
            
One can remark that since $\varphi$ has been made
monic, the determinant of $M$ is $\det(M) = p^n$. The first row of
$\LLL(M)$ gives a polynomial $g$ of degree at most $d^{'}$
that shares the common factor $\varphi$ of degree $n$ modulo $p$ with
$f$. 
The coefficients of $g$ have approximately a size $(\det M)^{1/(d^{'}+1)} =
p^{n/(d^{'}+1)}$ if we assume that the dimension stays small.

Note that, when $n=1$, this method produces the same pair $(f,g)$ as the
method of Joux-Lercier. Indeed, in this case $\varphi=x-m$ and the rows of
the matrix $M$ in Equation~\eqref{eq:LLL} generate the
same lattice as the rows of matrix $M$ in 
Equation~\eqref{eq:LLL generalization}. 

\begin{remark}
    By considering a smaller matrix, it is possible to produce a
    polynomial $g$ whose degree is smaller than $d' = \deg f-1$. 
    This does not seem to be a good idea. Indeed, the size of the
    coefficients of $g$ would be the same as the coefficients of a polynomial
	obtained starting with a polynomial $f$ with coefficients of the same
	size but of a smaller degree ($d'$ or less).
\end{remark}

We now discuss the criteria to select the parameter $d'=\deg f-1$ 
with respect to the bitsize of $p$. The most important quantity to
minimize is the size of the product of the norms
$\Reslt(\phi,f)\Reslt(\phi,g)$, for the typical polynomials $\phi$ that
will be used. In this setting, the best choice is to
stick to polynomials $\phi$ of degree 1, and we denote by $E$ a bound on
its two coefficients that we will tune later.
For any polynomial $P$, let us denote by
$\norm{P}_\infty$ the maximum absolute value of the coefficients of $P$. 
Since $f$ has been selected to have small coefficients, we obtain the
following estimate for the product of the norms:
\begin{equation}
    |\Reslt(\phi,f)\Reslt(\phi,g)| \approx
       \left(E^{\deg(f)}\right)\left(
       ||g||_\infty E^{\deg(g)}\right),
\end{equation}
where we did not write factors that contribute in a negligible way.
In Table~\ref{tab: ||g||oo wrt deg f, g, varphi, Joux-Lercier}, we list
the possible choices for the degrees, that we expect to be practically
relevant for discrete logarithms in $\F_{p^2}$ and~$\F_{p^3}$.

\begin{table}[htb] \caption{Size of the product of the norms, for various
        choices of parameters with the generalized Joux-Lercier method,
    in $\F_{p^2}$ and $\F_{p^3}$.} 
\label{tab: ||g||oo wrt deg f, g, varphi, Joux-Lercier} 
\centering 
\begin{tabular}{|c|c|c|c|c|c|} \hline
    Field & $\deg \varphi$ & $\deg f$ &
               $\deg g$ & $||g||_\infty$ & $E^{\deg f}E^{\deg g}||g||_\infty$\\
    \hline \hline
    $\F_Q = \F_{p^{2}}$ & 2 & 4 & 3 & $p^{1/2} = Q^{1/4}$ & $Q^{1/4} E^{7} $ \\
                      & 2 & 3 & 2 & $p^{2/3} = Q^{1/3}$ & $Q^{1/3} E^{5} $ \\ 
    \hline \hline     & 3 & 6 & 5 & $p^{3/6} = Q^{1/6}$ & $Q^{1/6} E^{11} $ \\
    $\F_Q = \F_{p^{3}}$ & 3 & 5 & 4 & $p^{3/5} = Q^{1/5}$ & $Q^{1/5} E^{9} $ \\
                      & 3 & 4 & 3 & $p^{3/4} = Q^{1/4}$ & $Q^{1/4} E^{7} $ \\
    \hline 
\end{tabular}
\end{table}

As for the value of the parameter $E$, although the asymptotic complexity
analysis can give hints about its value, it is usually not reliable for
fixed values. Therefore we prefer to use a rough approximation of
$E$ using the values of the same parameter in the factoring variant of
NFS as implemented in CADO-NFS. These values of $E$ w.r.t. $Q$ are
collected in Table~\ref{tab: E values for various Q} and we will use them
together with Table~\ref{tab: ||g||oo wrt deg f, g, varphi, Joux-Lercier}
in order to plot the estimate of the running time in
Figure~\ref{fig:norm-product-f-g--wrt-Q--Fpn}
to compare with other methods.  Note that, a
posteriori, the norms product in our case is smaller than in the
factoring variant of NFS. Hence, one can take a slightly smaller values for $E$.

\begin{table}[htb] \caption{Practical values of $E$ for $Q$ from 60 to
    220 decimal digits.}
 \label{tab: E values for various Q} 
$$\begin{array}{|c||ccccccccc|} \hline 
  Q (\mbox{dd}) &  60 &  80 & 100 & 120 & 140 & 160 & 180 & 204 & 220 \\
Q (\mbox{bits}) & 200 & 266 & 333 & 399 & 466 & 532 & 598 & 678 & 731 \\ 
\hline 
E (\mbox{bits}) &  19 &  20 &  21 &  23 &  25 &  27 &  28 &  29 &  30\\ 
\hline
\end{array}$$
 \end{table}

\subsection{The conjugation method} \label{subsec: our polyselect}

We propose another method to select polynomials for solving discrete
logarithms in $\F_{p^n}$ with the following features: the resulting
polynomial $f$ has degree $2n$ and small coefficients, while the
polynomial $g$ has degree $n$ and coefficients of size bounded by about
$\sqrt{p}$. In the next section, an asymptotic analysis shows that there
are cases where this is more interesting than the generalized
Joux-Lercier method; furthermore, it is also well suited for small degree
extension that can be reached with the current implementations.

Let us take an example.
\begin{example}
We take the case of $n=11$ and $p = 134217931$,  which is a random prime
congruent to $1$ modulo $n$. The method is very general, this case is the
simplest. We enumerate the integers $a=1,2,\ldots$ until $\sqrt{a}$ is
irrational but exists in
$\F_p$, i.e.~the polynomial $x^2-a$ splits in $\F_p$. We call $\lambda$ a square
root of $a$ in $\F_p$ and test if $x^n-\lambda$ is irreducible modulo $p$. If it
is not the case, we continue and try the next value of $a$. For example $a=5$
works.

 Then we set $\lambda= 108634777=\F_p(\sqrt{5})$ and we
put $\varphi=x^{11}-\lambda$. Next, we do a rational reconstruction of
$\lambda$, i.e.~we find two integers of size $O(\sqrt{p})$ such that
$u/v\equiv\lambda\mod p$. We find $u=1789$ and $v=10393$. The conjugation method consists in setting:
\begin{enumerate}
\item $f=(x^{11}-\sqrt{5})(x^{11}+\sqrt{5})=x^{22}-5$;
\item $g=vx^{11}-u=10393x^{11}-1789$. 
\end{enumerate}
By construction $f$ and $g$ are divisible by $\varphi$ modulo $p$.
\end{example}

We continue with a construction that works for $\F_{p^2}$ when $p$ is
congruent to 7 modulo 8. 

\begin{example}
Let $p \equiv 7 \bmod 8$ and let $f = x^4 + 1$ that is irreducible modulo
$p$. From the results in Section \ref{sec:exdeg4}, we use
the factorization $(x^2 + \sqrt{2}x +1) (x^2 - \sqrt{2}x +1)$ of
$f(x)$. Since $2$ is a square modulo $p$, we take $\varphi = x^2 +
\sqrt{2}x + 1 \in \Fp[x]$. Now, by rational reconstruction of
$\sqrt{2}$ in $\F_p$, we can obtain two integers $u,v \in \ZZ$ such
that $\frac{u}{v}\equiv \sqrt{2}\mod p$, and $u$ and $v$ have size
similar to $\sqrt{p}$.  We define $g = v x^2 + ux + v$. Then $f$ and
$g$ share a common irreducible factor of degree 2 modulo $p$, and
verify the degree and size properties that we announced.
\end{example}

This construction can be made general: first, it is possible to obtain
pairs of polynomials $f$ and $g$ with the claimed degree and size
properties for any extension field $\F_{p^n}$; and second, in many small
cases that are of practical interest, it is also possible to enforce the
presence of automorphisms.
The general construction is based on Algorithm~\ref{alg:conjugation}.

\begin{algorithm}[htb]
\caption{Polynomial selection with the conjugation method}
\label{alg:conjugation}
\KwIn{ $p$ prime and $n$ a small exponent}
\KwOut{ $f,g\in\Z[x]$ suitable for discrete logarithm computation with
NFS in $\F_{p^n}$}

Select $g_u(x), g_v(x)$, two polynomials with small integer coefficients,
$\deg g_u < \deg g_v = n$ \;

\Repeat {$\mu(x)$ has a root $\lambda$ in $\F_p$ and $g_v+\lambda g_u$ is
irreducible in $\F_p$}
{Select $\mu(x)$ a quadratic, monic, irreducible polynomial over
$\Z$ with small coefficients \;}
 $(u,v)\gets$ a rational reconstruction of $\lambda$ \;
 $f\gets \Reslt_Y(\mu(Y), g_v(x) + Yg_u(x))$ \;
 $g\gets vg_v +u g_u $ \;
 \Return {$(f,g)$}
\end{algorithm}

\begin{fact}[Properties of the conjugation method]
    The polynomials $(f,g)$ returned by Algorithm~\ref{alg:conjugation}
    verify:
    \begin{enumerate}
        \item $f$ and $g$ have integer coefficients and degrees $2n$ and
            $n$ respectively;
        \item the coefficients of $f$ have size $O(1)$ and the
            coefficients of $g$ are bounded by $O(\sqrt{p})$.
        \item $f$ and $g$ have a common irreducible factor $\varphi$ of
            degree $n$ over $\F_p$.
    \end{enumerate}
\end{fact}

\begin{proof}
    The fact that $g$ has integer coefficients and is of degree $n$ is
    immediate by construction. As for $f$, since it is the resultant of
    two bivariate polynomials with integer coefficients, it is also with
    integer coefficients. Using classical properties of the resultant,
    $f$ can be seen as the product of the polynomial $ g_v(x) + Yg_u(x)$
    evaluated in $Y$ at the two roots of $\mu(Y)$, therefore its degree
    is $2n$. Also, since all the coefficients of the polynomials involved
    in the definition of $f$ have size $O(1)$, and the degree $n$ is
    assumed to be ``small'', then the coefficients of $f$ are also~$O(1)$.

    For the size of the coefficients of $g$, it follows from the output
    of the rational reconstruction of $\lambda$ in $\F_p$, which is
    expected to have sizes in $O(\sqrt{p})$ (in theory, we can not
    exclude that we are in a rare case where the sizes are larger,
    though).

    The polynomials $f$ and $g$ are suitable for NFS in $\F_{p^n}$,
    because both are divisible by $\varphi = g_v+\lambda g_u$ modulo $p$,
    and by construction it is irreducible of degree $n$.
\end{proof}

In the example above, for $\F_{p^2}$ with $p\equiv 7 \mod 8$,
Algorithm~\ref{alg:conjugation} was applied with 
$g_u = x$, $g_v = x^2+1$ and $\mu = x^2-2$. One can check that $f =
\Reslt_Y(Y^2-2, (x^2+1) + Yx) = x^4+1$, as can be seen from Section
\ref{sec:exdeg4}.
\bigskip

In Algorithm~\ref{alg:conjugation}, there is some freedom in the choices
of $g_u$ and $g_v$. The key idea to exploit this opportunity is to base
the choice on one-parameter families of polynomials for which an
automorphism with a nice form is guaranteed to exist, in order to use the
improvements of Section~\ref{sec:galois}.

In Table~\ref{tab:foster}, we list possible choices for $g_u$ and $g_v$
in degree $2$, $3$, $4$ and $6$, such that for any integer $\lambda$,
$g_v+\lambda g_u$ as a simple explicit cyclic automorphism. The
families for $3$, $4$ and $6$ are taken from~\cite{Gras79,Gras87} (see
also \cite{Fos11} references for larger degrees).

\begin{table} 
\begin{center}
\caption{Families of polynomials of degree 2, 3, 4 and 6 with
cyclic Galois group.}
\label{tab:foster} 
$
\begin{array}{|c|c|c|c|c|} 
\hline n & \text{coeffs of $g_v+ag_u$} & g_v & g_u & 
                                           \text{automorphism: } \theta\mapsto \\
\hline 
\hline   & (1,a,1)  & x^2 + 1 & x &  1/\theta \\
       2 & (-1,a,1) & x^2 - 1 & x & -1/\theta \\
         & (a,0,1)  & x^2     & 1 & -  \theta \\
\hline 3 & (1,-a-3,-a,1) & x^3-3x-1 & -(x^2+x)& -(\theta+1)/\theta \\
\hline 4 & (1,-a,-6,a,1) & x^4-6x^2+1 & x^3-x   & -(\theta+1)/(\theta-1) \\
\hline 6 & \begin{array}{l}(1,-2a,-5a-15,\\
\quad-20,5a,2a+6,1) 
\end{array}& \begin{array}{l}x^6+6x^5-\\ \quad20x^3-15x^2+1\end{array} &
\begin{array}{l}2x^5+5x^4-\\ \quad5x^2-2x\end{array} &-(2\theta+1)/(\theta - 1) \\
\hline
\end{array} 
$
\end{center}
\end{table}

\begin{theorem}
    For any prime $p$ and $n$ in $\{2,3,4,6\}$, the polynomials $f$ and
    $g$ obtained by the conjugation method using $g_u$ and $g_v$ as in
    Table~\ref{tab:foster} generate number fields with two automorphisms
    $\sigma$ and $\tau$ of order $n$ that verify the hypothesis of
    Theorem~\ref{th:galois}.
\end{theorem}

\begin{proof}
    The polynomial $g$ belongs to a family of Table~\ref{tab:foster},
    so its number field $K_g$ has an automorphism of order~$n$ given by the
    formula in the last column. 
    
    Let $\omega$ be a root of $\mu(x)$. The polynomial 
    $g_v+\omega g_u$ defines a number field that is an extension of
    degree $n$ of $\Q(\omega)$ and that admits an automorphism of order
    $n$, which fixes $\Q(\omega)$.  Since $f$ and $g_v+\omega g_u$
    generate the same number field, this shows that
    the number field $K_f$ has an automorphism of order $n$.

    The polynomial $\varphi$ is given by $g_v+\lambda g_u$. Therefore, it
    belongs to the same family as $g$ hence it has the same automorphism
    of order $n$ as $f$ and $g$. This shows that modulo $p$, the
    automorphism sends a root of $\varphi$ to another root of $\varphi$,
    as required in the hypothesis of Theorem~\ref{th:galois}.
\end{proof}

\begin{example} Let us apply the conjugation method for $\F_{p^3}$, where
    $p=2^{31}+11$.  Running Algorithm~\ref{alg:conjugation} with
    $g_u=-x^2-x$ and $g_v=x^3-3x-1$, one sees that $\mu = x^2-x+1$ has a
    root $\lambda=2021977950$ in $\F_p$ and that $g_v+\lambda g_u$ is
    irreducible in $\F_p[x]$. We obtain $f = x^6 - x^5 - 6x^4 + 3x^3 + 14x^2
    + 7x + 1$ and $g = 20413x^3 + 32630x^2 -28609x + 20413$. With 
    $\varphi = x^3 + 125505709x^2 + 125505706x + 2147483658$ as their GCD
    modulo $p$, we can
    check that the three polynomials $f$, $g$ and $\varphi$ admit
    $\theta\mapsto -(\theta+1)/\theta$ as an automorphism of order 3.
\end{example}

\subsection{Estimation and comparison of the methods.} 
\label{subsec:comparison of gal Joux Lercier and ours}
We have four methods of polynomial selection which apply to NFS in non-prime
fields:
\begin{itemize}
\item the two methods of JLSV, presented in~\ref{sssec:JLSV1} and
\ref{sssec:JLSV2}, denoted JLSV$_1$
and, respectively, JLSV$_2$;
\item the generalized Joux-Lercier method, presented in~\ref{subsec: Joux Lercier polyselect}, denoted GJL;
\item the conjugation method, presented in~\ref{subsec: our polyselect}, denoted bu Conj. 
\end{itemize}
We take the size of the product of the norms as the main quantity to
minimize, and we estimate its value~as
\begin{equation} \label{eq:norm-product-f-g} 
 E^{\deg f}||f||_\infty E^{\deg g} ||g||_\infty  ~.
\end{equation}

The starting point are the properties of the polynomials obtained with the
various methods in Tab.~\ref{tab: polyselect complexities}. 
\begin{table}[htb]
\caption{Theoretical complexities for polynomial selection methods, 
$n$ is the extension degree ($\F_{p^{n}}$), $Q = p^n$}
\label{tab: polyselect complexities}
\begin{tabular}{|l | c | c | c | c | c|}
\hline      method & $\deg g$ & $\deg f$    & $||g||_\infty$                & $||f||_\infty$ & $E^{\deg f + \deg g} ||f||_\infty ||g||_\infty$ \\
\hline Conj        & $n$      & $2n$        & $Q^{1/(2n)}$          & $O(1)$ & $E^{3n} Q^{1/(2n)}$ \\
\hline GJL         & $\geq n$ &$> \deg g$   & $Q^{1/(\deg g+1)}$ & $O(1)$ & $E^{\deg f + \deg g} Q^{1/(\deg g+1)} $ \\
\hline JLSV$_1$    & $n$      & $n$         & $Q^{1/(2n)}$ &  $Q^{1/(2n)}$ & $E^{2n} Q^{1/n}$ \\
\hline JLSV$_2$    & $n$      &$\geq \deg g$&  $Q^{1/(2(\deg f+1))}$ &  $Q^{1/(\deg f+1)}$ & $E^{\deg f + n} Q^{(3/2) 1/(\deg f+1)} $\\
\hline
\end{tabular}
\end{table}

When the best method depends on the size of the finite field in consideration,
we use rough estimates of $E$ taken from Table~\ref{tab: E values for various Q}.

In Table~\ref{tab:norm-product-f-g-wrt-Q-E} we summarize all the sizes
that we can get for reasonable choices of parameters for $\F_{p^n}$ with
$n\in\{2,3,4,5,6\}$, with all the methods at our disposition.

\begin{table}[htbp] \caption{
        Size of the product of norms for various choices of parameters
        and algorithms. We discard ($\otimes$) the methods which offer sizes of norms product which are clearly not 
competitive compared to some other one, assuming that 
$0.04 \log Q \leq \log E \leq 0.1 \log Q$ 
(Tab.~\ref{tab: E values for various Q})}
\label{tab:norm-product-f-g-wrt-Q-E}
$$\begin{array}{|c| c||c|c|c|c|c r|}
 \hline 
\deg g, \deg f & Q & \F_{p^n} & ||f||_{\infty} & g & ||g||_{\infty} &
       \mc{2}{c|}{E^{\deg f} ||f||_{\infty} E^{\deg g} ||g||_{\infty}} \\
 \hline \hline
 (2,3) & \mrm{8}{p^2} & \mrm{8}{\Fps} & \mrm{3}{O(1)} & \GJL & Q^{1/3} & 
                                                                 E^5 Q^{1/3} & \\
 \cline{1-1} \cline{5-8} 
(3,4) & & &               & \GJL   & Q^{1/4} & E^7 Q^{1/4}  & \otimes \\
 \cline{1-1} \cline{5-8} 
(2,4) & & &               & \Conj  & Q^{1/4} & E^6 Q^{1/4}  & \\
 \cline{1-1} \cline{4-8} 
(2,2) & & &  Q^{1/4}      & \JLSV_1 & Q^{1/4} & E^4 Q^{1/2}  & \otimes \\
 \cline{1-1} \cline{4-8} 
(2,2) & & &  Q^{1/6}      & \JLSV_2 & Q^{1/3} & E^4 Q^{1/2}  & \otimes \\
 \cline{1-1} \cline{4-8} 
(2,3) & & &  Q^{1/8}      & \JLSV_2 & Q^{1/4} & E^5 Q^{3/8} &  \otimes \\
 \cline{1-1} \cline{4-8} 
(2,4) & & &  Q^{1/5}      & \JLSV_2 & Q^{1/10} & E^6 Q^{3/10} &  \otimes \\
 \cline{1-1} \cline{4-8} 
(2,5) & & &  Q^{1/6}      & \JLSV_2 & Q^{1/12} & E^7 Q^{1/4} &  \otimes \\
 \hline \hline
 (3,4) & \mrm{8}{p^3} & \mrm{8}{\Fpc} & \mrm{3}{O(1)} & \GJL & Q^{1/4} & E^7 Q^{1/4} & \\
 \cline{1-1} \cline{5-8} 
(4,5) & & &               & \GJL  & Q^{1/5} & E^9 Q^{1/5}  & \otimes \\
 \cline{1-1} \cline{5-8} 
(3,6) & & &               & \Conj  & Q^{1/6} & E^9 Q^{1/6}  & \\
 \cline{1-1} \cline{4-8} 
(3,3) & & & Q^{1/6}       & \JLSV_1 & Q^{1/6} & E^6 Q^{1/3} & \otimes \\
 \cline{1-1} \cline{4-8} 
(3,3) & & & Q^{1/8}       & \JLSV_2 & Q^{1/4} & E^6 Q^{3/8} & \otimes \\
 \cline{1-1} \cline{4-8} 
(3,4) & & & Q^{1/10}      & \JLSV_2 & Q^{1/5} & E^7 Q^{3/10} & \otimes \\
 \cline{1-1} \cline{4-8} 
(3,5) & & & Q^{1/12}      & \JLSV_2 & Q^{1/6} & E^8 Q^{1/4} & \otimes \\
 \cline{1-1} \cline{4-8} 
(3,6) & & & Q^{1/7}      & \JLSV_2 & Q^{1/14} & E^9 Q^{3/14} & \otimes \\
 \hline \hline
 (4,5) & \mrm{8}{p^4} & \mrm{8}{\F_{p^{4}}} & \mrm{3}{O(1)} & \GJL & Q^{1/5} & E^9 Q^{1/5} & \\
 \cline{1-1} \cline{5-8} 
 (5,6) & & &              & \GJL    & Q^{1/6} & E^{11} Q^{1/6} & \otimes \\
 \cline{1-1} \cline{5-8} 
 (4,8) & & &              & \Conj   & Q^{1/8} & E^{12} Q^{1/8} & \otimes \\
 \cline{1-1} \cline{4-8} 
 (4,4) & & & Q^{1/8}      & \JLSV_1  & Q^{1/8} & E^8 Q^{1/4} & \\
 \cline{1-1} \cline{4-8} 
 (4,4) & & & Q^{1/10} & \JLSV_2  & Q^{1/5} & E^8 Q^{3/10} & \otimes \\
 \cline{1-1} \cline{4-8} 
 (4,5) & & & Q^{1/12} & \JLSV_2  & Q^{1/6} & E^9 Q^{1/4} & \otimes \\
 \cline{1-1} \cline{4-8} 
 (4,6) & & & Q^{1/14} & \JLSV_2  & Q^{1/7} & E^{10} Q^{3/14} & \otimes \\ 
 \cline{1-1} \cline{4-8} 
 (4,7) & & & Q^{1/8}  & \JLSV_2  & Q^{1/16} & E^{11} Q^{3/16} & \otimes \\
 \hline \hline
 (5,6) & \mrm{8}{p^5} & \mrm{8}{\F_{p^{5}}} & \mrm{3}{O(1)} & \GJL & Q^{1/6} & E^{11} Q^{1/6} & \\
 \cline{1-1} \cline{5-8} 
 (6,7) & & &              & \GJL   & Q^{1/7} & E^{13} Q^{1/7} & \otimes \\
 \cline{1-1} \cline{5-8} 
 (5,10) & & &              & \Conj   & Q^{1/10} & E^{15} Q^{1/10} & \otimes \\
 \cline{1-1} \cline{4-8} 
 (5,5) & & & Q^{1/10}      & \JLSV_1  & Q^{1/10} & E^{10} Q^{1/5} & \\
 \cline{1-1} \cline{4-8} 
 (5,5) & & &  Q^{1/12} & \JLSV_2  & Q^{1/6} & E^{10} Q^{1/4} & \otimes \\
 \cline{1-1} \cline{4-8} 
 (5,6) & & & Q^{1/14} & \JLSV_2  & Q^{1/7} & E^{11} Q^{3/14} & \otimes \\
 \cline{1-1} \cline{4-8} 
 (5,7) & & & Q^{1/8}  & \JLSV_2  & Q^{1/16} & E^{12} Q^{3/16} & \otimes \\
 \cline{1-1} \cline{4-8} 
 (5,8) & & & Q^{1/18} & \JLSV_2  & Q^{1/9} & E^{13} Q^{1/6} & \otimes \\
 \hline \hline
 (6,7) & \mrm{8}{p^6} & \mrm{8}{\F_{p^{6}}} & \mrm{3}{O(1)} & \GJL & Q^{1/7} & E^{13} Q^{1/7} & \\
 \cline{1-1} \cline{5-8} 
 (7,8) & & &              & \GJL     & Q^{1/8}  & E^{15} Q^{1/8} & \otimes \\
 \cline{1-1} \cline{5-8} 
 (6,12) & & &              & \Conj   & Q^{1/12} & E^{18} Q^{1/12} & \otimes \\
 \cline{1-1} \cline{4-8} 
 (6,6) & & & Q^{1/12}      & \JLSV_1  & Q^{1/12} & E^{12} Q^{1/6} & \\
 \cline{1-1} \cline{4-8} 
 (6,6) & & & Q^{1/14}      & \JLSV_2  & Q^{1/7} & E^{12} Q^{3/14} & \otimes \\
 \cline{1-1} \cline{4-8} 
 (6,7) & & & Q^{1/16}      & \JLSV_2  & Q^{1/8} & E^{13} Q^{3/16} & \otimes \\
 \cline{1-1} \cline{4-8} 
 (6,8) & & & Q^{1/18}      & \JLSV_2  & Q^{1/9} & E^{14} Q^{1/6} & \otimes \\
 \cline{1-1} \cline{4-8} 
 (6,9) & & & Q^{1/20}      & \JLSV_2  & Q^{1/10} & E^{15} Q^{3/20} & \otimes \\
 \hline 
\end{array}$$
 \end{table}
 To
choose the best method, we now need to compare the values in the last
column of 
Tab.~\ref{tab:norm-product-f-g-wrt-Q-E}.  
For that we consider in Tab.~\ref{tab: E values for various Q}
practical values of $E$ and $Q$ for $Q$ from 60 to 220 decimal digits~(dd). 
We note that $\log E = 0.095 \log Q$ for $Q$ of $60$~dd and 
$\log E = 0.041 \log Q$ for $Q$ of $220$~dd. We can now eliminate a few other 
methods in Tab.~\ref{tab:norm-product-f-g-wrt-Q-E}:
\begin{description}
 \item[$\boldsymbol{n=2}$] We discard the $\JLSV_1$ and  $\JLSV_2$ methods because 
 their complexities are worse than GJL complexity: $E^{4}Q^{1/2} > E^5
 Q^{1/3}$ since $Q^{1/6} > E$ (indeed, $Q^{0.1} > E$).
 \item[$\boldsymbol{n=3}$] A second time we discard the $\JLSV_1$ method because
 the GJL method is better.
Indeed $E^6 Q^{1/3} < E^{7} Q^{1/4}$ while $E > Q^{1/12}$ i.e.~when
 $Q$ is less or around 60 dd.
 \item[$\boldsymbol{n=4}$] This time we discard GJL method with $(\deg g, \deg f) =
 (5,6)$ because it is less efficient than GJL with $(4,5)$ whenever $E
 > Q^{1/60}$.
We also discard the Conj method because we are not in the case $E < Q^{1/40}$. 
 \item[$\boldsymbol{n=5}$] We discard the Conj method ($E^{15} Q^{1/10}$) which is
 worse than
   GJL with $(5,6)$ while $ E > Q^{1/60}$. We also discard GJL method with 
  $(6,7)$ because the same method with $(5,6)$ is more efficient whenever 
  $Q^{1/84} < E$.
 \item[$\boldsymbol{n=6}$] As for $n=5$, the Conj method is not competitive because
 we are not in the case 
 $E < Q^{1/84}$. We also discard the GJL method with $(7,8)$ compared
 with $(6,7)$ because 
 we don't have $E < Q^{1/112}$.
\end{description}

We represent the results in 
Fig.~\ref{fig:norm-product-f-g--wrt-Q--Fpn}. 
We can clearly see that when $Q = p^2$ is more than 70 decimal digits long
(200 bits, i.e.~$\log_2 p = 100$), it is much better to use the
construction with $\deg f = 4$ and $\deg g = 2$ for computing discrete
logarithms in $\F_{Q} = \Fps$. For $Q$ of more than 220 dd, 
$(\deg f, \deg g) =
(3,4)$ starts to be a better choice than $(2,3)$ but our new method with
$(2,4)$ is even better, the value 
in \eqref{eq:norm-product-f-g} 
is about 20 bits smaller.  For $Q = p^3$ from 60
to 220~dd (i.e.~$p$ from 20 to 73~dd), the choice $(3,4)$ gives a lower value
of \eqref{eq:norm-product-f-g}. Then for
$Q$ of more than 220~dd, the method with $(3,6)$ is better. For $Q $ of 220~dd, 
\eqref{eq:norm-product-f-g} takes the 
same value with $(\deg g, \deg f) = (3,6)$ as with $(3,4)$.

\begin{figure}[htbp]
\centering
\subfigure[$Q=p^4,p^5,p^6$: JLSV$_1$ or GJL method]{
  \includegraphics{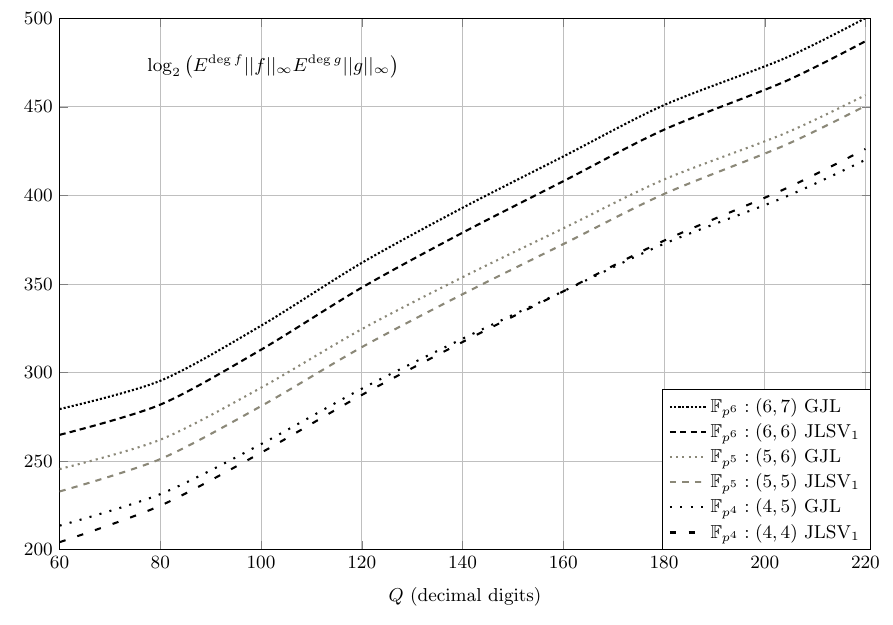} 
}
\subfigure[$Q=p^2$ or $p^3$: Conjugation or GJL method]{
  \includegraphics{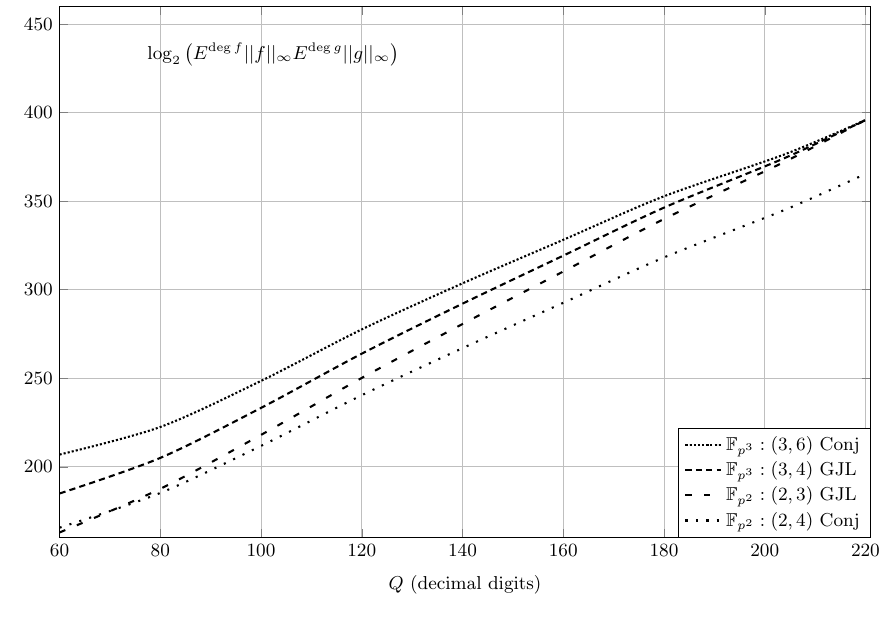} 
}
\caption{Estimation of
\eqref{eq:norm-product-f-g} for various
pairs $(\deg f, \deg g)$ selected with our two methods for computing
discrete logarithms in $\F_{p^n}$ with $n \in \{2,3,4,5,6\}$.} 
\label{fig:norm-product-f-g--wrt-Q--Fpn}
\end{figure}

\subsection{Improving the selected polynomials} 
\label{subsec: improved g in polyselect}

We explained in Sec.~\ref{subsec: Joux Lercier polyselect} our generalized
Joux-Lercier method and in Sec.~\ref{subsec: our polyselect} our method of
conjugated polynomials. In both cases when $\deg \varphi \geqslant 2$ one
obtains two distinct reduced polynomials $g_1$ and $g_2 \in \ZZ[x]$ such
that $g_1 \equiv g_2 \equiv \varphi \bmod p$ up to a coefficient in $\Fp$.
We propose to search for a polynomial $g = \lambda_1 g_1 + \lambda_2 g_2$
with $\lambda_1, \lambda_2 \in \ZZ$ small, e.g. $|\lambda_1|, |\lambda_2| <
200$ that maximises the Murphy $E$ value of the pair $(f,g)$. 

The Murphy $\EE$ value is explained in \cite[Sec.~5.2.1, Eq.~5.7
p.~86]{Mur99}.  This is an estimation of the smoothness properties of the
values taken by either a single polynomial $f$ of a pair $(f,g)$. First one
homogenizes $f$ and $g$ and defines $$ u_f(\theta_i) = \frac{\log |f(\cos
\theta_i, \sin \theta_i)| + \alpha(f)}{\log B_f}  $$ with $\theta_i \in
\left[ 0, \pi \right]$, more precisely, $\theta_i = \frac{\pi}{K} \left(i -
\frac{1}{2}\right) $ (with e.g.  $K = 2000$ and $i \in \{1, \ldots, K \}$),
$\alpha(f)$ defined in \cite[Sec.~3.2.3]{Mur99} and $B_f$ a smoothness
bound set according to $f$.  Murphy advises to take $B_f = 1\e7$ and $B_g =
5\e6$.  Finally $$ \EE(f,g) = \sum_{i=1}^{K} \rho (u_f(\theta_i)) \rho (u_g
(\theta_i)) ~.$$

We propose to search for $g = \lambda_1 g_1 + \lambda_2 g_2$ with
$|\lambda_i| < 200$ and such that $\EE(f,g)$ is maximal. In practice we
obtain $g$ with $\alpha(g) \leqslant -1.5$ and $\EE(f,g)$ improved of 2\%
up to 30 \%.

\section{Asymptotic complexity}
\label{sec:complexity}
The two new methods of polynomial selection require a dedicated analysis
of complexity. First, we show that the generalized Joux-Lercier method
offers an alternative to the existing method of polynomial selection in
large characteristic~\cite{JLSV06} and determine the range of
applicability in the boundary case. When getting close to the limit, it
provides the best known complexity. Second, we analyze the conjugation
method and obtain the result announced in the introduction, namely the
existence of a family of finite fields for which the complexity of
computing discrete logarithms is in $L_Q(1/3,\sqrt[3]{48/9})$.

\subsection{The generalized Joux-Lercier method} Using the generalized
Joux-Lercier method, one constructs two polynomials $f$ and $g$ such that,
for a parameter $d\geq n$, we have $\deg f=d+1$, $\deg g=d$,
$\norm{g}_\infty=Q^{1/d}$ and $\norm{f}_\infty$ is very small, say $O(\log
Q)$. 

We consider the variant of NFS in which one sieves on linear polynomials
$a-bx$ such that $|a|,|b|\leq E$ for a sieve parameter $E$, in order to
collect the pairs such that the norms $\Reslt(f,a-bx)$ and $\Reslt(g,a-bx)$ are
$B$-smooth.

Since the cost of the sieve is $E^{2+o(1)}$ and the cost of the linear
algebra stage is $B^{2+o(1)}$, we impose $E=B$. We set $E=B=L_Q(1/3,\beta)$
for a parameter $\beta$ to be chosen. We write
$d=\frac{\delta}{2}\left(\log Q/\log\log Q\right)^{1/3}$, for a parameter
$\delta$ to be chosen.

Since the size of the sieving domain must be large enough so that we
collect $B$ pairs $(a,b)$, we must have $\mathcal{P}^{-1}=B$, where
$\mathcal{P}$ is the probability that a random pair $(a,b)$ in the sieving
domain has $B$-smooth norms. We make the usual assumption that the product
of the norms of any pair $(a,b)$ has the same probability to be $B$-smooth
as a random integer of the same size. We upper-bound the norms product by
\begin{equation} |\Reslt(f,a-bx)\Reslt(g,a-bx)|\leq (\deg f)\norm{f}_\infty
E^{\deg f}(\deg g)\norm{g}_\infty E^{\deg g}, \end{equation} and further,
with the $L$-notation, we obtain \begin{equation}
|\Reslt(f,a-bx)\Reslt(g,a-bx)|\leq L_Q\left(2/3, \delta\beta +\frac{2}{\delta}
\right).  \end{equation} Using the Canfield-Erdös-Pomerance theorem, we
obtain \begin{equation} \mathcal{P}=1/L_Q\left(1/3,
\frac{\delta}{3}+\frac{2}{3\beta\delta}\right).  \end{equation} The
equality $\mathcal{P}^{-1}=B$ imposes \begin{equation}
\beta=\frac{\delta}{3}+\frac{2}{3\beta\delta}.  \end{equation} The optimal
value of $\delta$ is the one which minimizes the expression in the right
hand member, so we take $\delta=\sqrt{2/\beta}$ and we obtain
$\beta=2/3\sqrt{2/\beta}$, or equivalently $\beta=\sqrt[3]{8/9}$. Since the
complexity of NFS is $E^2+B^2=L_Q(1/3,2\beta)$, we obtain \begin{equation}
    \text{complexity}(\text{NFS with Generalized Joux-Lercier})
=L_Q\left(1/3,\sqrt[3]{64/9}\right).  \end{equation}

The method requires $n\leq d$. Since $d=\delta/2 \left(\frac{\log Q}{\log\log
Q} \right)^{1/3}$ with $\delta=\sqrt{2/\beta}=\sqrt[3]{3}$, the method
applies only to fields $\F_{p^n}$ such that \begin{equation} p\geq
L_Q\left(2/3,\sqrt[3]{8/3}\right).  \end{equation}

\subsection{The conjugation method} The conjugation method allows us to
construct two polynomials $f$ and $g$ such that $\deg f=2n$, $\deg g=n$,
$\norm{g}_\infty\approx p^{1/2}$ and $\norm{f}_\infty$ is very small, say $O(\log
Q)$. We study first the case of medium characteristic and then the boundary
 case between medium and large characteristic. We start with those
computations which are common for the two cases.

\subsubsection{Common computations}
We consider the higher degree variant of NFS of parameter $t$, i.e.~one
sieves on polynomials $\phi$ of degree $t-1$, with coefficients of absolute
value less than $E^{2/t}$, where $E$ is called the sieve parameter. The
cost of the sieve is then $E^{2+o(1)}$. Since cost of the
linear algebra stage is $B^{2+o(1)}$, where $B$ is the smoothness bound, we
impose $E=B$ and we write $E=B=L_Q(1/3,\beta)$, for some parameter $\beta$
to be chosen.  Then the product of the norms of $\phi(\alpha)$ and
$\phi(\beta)$ for any polynomial $\phi$ in the sieve domain is
\begin{equation*} | \Reslt(\phi,f)\Reslt(\phi,g) | \leq (\deg f+t)!(\deg g+t)!
E^{4n/t}\norm{f}_\infty^{t-1}E^{2n/t}\norm{g}_\infty^{t-1}.
\end{equation*} Since $(\deg f+t)!(\deg f+t)!\leq L_Q(2/3, o(1))$, this
factor's contribution will be negligible compared to the main term which
is in $L_Q(2/3)$. Therefore we have
\begin{equation*} 
| \Reslt(\phi,f)\Reslt(\phi,g) | \leq \left(
E^{6n/t}Q^{(t-1)/2n}\right)^{1+o(1)}.
  \end{equation*} We make the usual assumption that the norms product has
the same probability to be $B$-smooth as a random integer of the same size.  

\subsubsection{The medium characteristic case}
Let us set the value of the number of terms in the sieve:
\begin{equation}
t=c_t n \left(\frac{\log Q}{\log\log Q} \right)^{-1/3}.
\end{equation}
The probability that a polynomial $\phi$ in the sieving domain has
$B$-smooth norms is
\begin{equation}
\mathcal{P}=1/L_Q\left(1/3,\frac{2\beta}{c_t}+\frac{c_t}{6} \right).
\end{equation}
We choose $c_t=2\sqrt{3\beta}$ in order to minimize the right hand member:
\begin{equation}
\mathcal{P}=1/L_Q\left(1/3,2\sqrt{\beta/3}\right).
\end{equation}

In an optimal choice of parameters, the sieve produces just enough
relations, so we require that $\mathcal{P}^{-1}=B$, and equivalently
$\beta=\sqrt[3]{4/3}$. We obtain
\begin{equation}
\text{complexity}(NFS\text{ with medium
char.})=L_Q\left(1/3,\sqrt[3]{96/9}\right).
\end{equation}

\subsubsection{The boundary case}
For every constant $c_p>0$, we consider the family of finite fields
$\F_{p^n}$ such that 
\begin{equation} 
p=L_{p^n}(2/3,c_p)^{1+o(1)}.
\end{equation}

The parameter $t$ is a constant, or equivalently we have a different
algorithm for each value $t=2,3,\ldots$.

Then the probability that a
polynomial $\phi$ in the sieving domain has $B$-smooth norms is
\begin{equation} \mathcal{P}=1/L_Q\left(1/3,
\frac{2}{c_pt}+\frac{c_p(t-1)}{6\beta}  \right).  \end{equation} 
If the parameters are tuned to have just enough relations in the sieve,
then one has $\mathcal{P}^{-1}=B$. This leads to
$\frac{2}{c_pt}+\frac{c_p(t-1)}{6\beta}=\beta$, or
$\beta=\frac{1}{c_pt}+\sqrt{\frac{1}{(c_pt)^2}+\frac16 c_p(t-1)}$. Hence,
the complexity of NFS with the conjugation method is: \begin{equation}
    \text{complexity(NFS with the conjugation method)}=L_Q\left(1/3,
\frac{2}{c_pt}+\sqrt{\frac{4}{(c_pt)^2}+\frac23c_p(t-1)}  \right).
\end{equation}

In Figure~\ref{fig:complexities}, we have plotted the complexities of
various methods, including the Multiple number field sieve variant
of~\cite{BarPie2014}. There are some ranges of the parameter $c_p$ where
our conjugation method is the fastest and a range where the generalized
Joux-Lercier method is the fastest. 
The best case for our new method corresponds to the case 
where $c_p = 12^{1/3}\approx 2.29$ and $t=2$. In that case we get:
\begin{equation}
\text{complexity}\text{(best
case for the conjugation method)}=L_Q\left(1/3,\sqrt[3]{\frac{48}{9}}\right).  \end{equation}

\begin{figure} \begin{center} \includegraphics{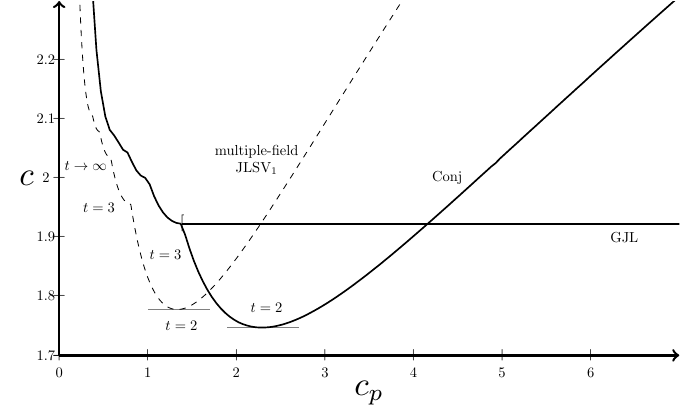}
\end{center}
\caption{The complexity of NFS for fields $\F_{p^n}$ with
$p=L_Q(2/3,c_p)$ is $L_Q(1/3,c)$. The blue curve corresponds to the
multiple number field sieve of~\cite{BarPie2014}, the green semi-line
to the generalized Joux-Lercier method and the red thick curve to the conjugation method.}
\label{fig:complexities} \end{figure}

\section{Effective computations of discrete logarithms}
\label{sec:effective}

In order to test how our ideas perform in practice, we did a medium-sized
practical experiment in a field of the form $\F_{p^2}$. Since we could
not find any publicly announced computation for this type of field, we
have decided to choose a prime number $p$ of 80 decimal digits so that 
$\F_{p^2}$ has size 160 digits. 
To demonstrate that our approach is not specific to a particular form of
the prime, we took the first 80 decimal digits of $\pi$. 
Our prime number $p$ is the next prime such that $p \equiv 7 \bmod 8$ and 
both $p+1$ and $p-1$ have a large prime factor: 
$p = \lfloor \pi \cdot 10^{79} \rfloor + 217518 $.

\begin{scriptsize}
$$\begin{array}{rcl}
p    & = & \mathtt{31415926535897932384626433832795028841971693993751058209749445923078164063079607} \\
\ell & = & \mathtt{3926990816987241548078304229099378605246461749218882276218680740384770507884951} \\
p-1  & = & 6 \cdot h_0 \mbox{ with } h_0 \mbox{ a 79 digit prime} \\
p+1  & = & 8 \cdot \ell \\
\end{array}$$
\end{scriptsize}
We tried to solve the discrete logarithm
problem in the order $\ell$ subgroup. We imposed $p$ to be
congruent to $-1$ modulo 8, so that the polynomial $x^4+1$ could be used,
as in Section~\ref{sec:exdeg4}, so that no Schirokauer map is
needed. The conjugation method yields a polynomial $g$ of degree $2$
and negative discriminant, a particular case that requires no
Schirokauer map either:
$$\begin{array}{rcl}
f & = & x^4 + 1 \\
g & = & 22253888644283440595423136557267278406930\ x^2 \\
  &   & \ +\, 41388856349384521065766679356490536297931\ \ x \\
  &   & \ +\, 22253888644283440595423136557267278406930\ \,. \\
\end{array}$$
Since $p$ is 80 digits long, the coefficients of $g$ have almost 40 digits 
(precisely 41 digits). The polynomials $f$ and $g$ have the irreducible factor

{\scriptsize
$$\varphi(t) = t^2 +
8827843659566562900817004173601064660843646662444652921581289174137495040966990\, t + 1$$}
in common modulo $p$, and $\GFpn{p}{2}$ will be taken as
$\GFq{p}[X]/(\varphi)$.

The relation collection step was then done using the sieving software of
CADO~\cite{CADO}. More precisely, we used the special-$\gq$ technique for
ideals $\gq$ on the $g$-side, since it produces norms that are larger
than on the $f$-side. We sieved all the special-$\gq$ larger than 
$40,000,000$ and smaller than $2^{27}$, keeping only one in each pair of
conjugates, as explained in Section~\ref{sec:galois}. In total, they
produced about $15$M relations. The main parameters in
the sieve were the following: we sieved all primes below $40$M, and we
allowed two large primes less than $2^{27}$ on each side. The search
space for each special-$\gq$ was set to $2^{15}\times 2^{14}$ (the
parameter {\tt I} in CADO was set to 15).

The total CPU time for this relation collection step is equivalent to 68
days on one core of an Intel Xeon E5-2650 at 2 GHz. This was run in
parallel on a few nodes, each with 16 cores, so that the elapsed time for
this step was a few days, and could easily be made arbitrary small with
enough nodes.

The filtering step was run as usual, but we modified it to take into
account the Galois action on the ideals: we selected a representative
ideal in each orbit under the action $x\mapsto x^{-1}$, and rewrote all
the relations in terms of these representatives only. This amounts just
to keep track of sign-change, that has to be reminded when combining two
relations during the filtering, and when preparing the sparse matrix for
the sparse linear algebra step. The output of the filtering step was a
matrix with $839,244$ rows and columns, having on average $83.6$ non-zero
entries per row.

Thanks to our choice of $f$ and $g$, it was not necessary to add
columns with Schirokauer maps. We used Jeljeli's implementation of
Block Wiedemann's algorithm for GPUs \cite{JeljeliImplem}. In fact, this was
a small enough computation so that we did not distribute it on several
cards: we used a non-blocked version. The total running time for this
step was around 30.3 hours on an NVidia GTX 680 graphic card.

At the end of the linear algebra we know the virtual logarithms of almost
all prime ideals of degree one above primes of at most 26 bits, and of
some of those above primes of 27 bits. At this point we could test that
the logs on the f-side were correct.

The last step is that of computing some individual logarithms.
We used $G = t + 2$ as a generator for $\GFpn{p}{2}$ and
the following ``random'' element:
$$s = \lfloor(\pi  (2^{264})/4)\rfloor t + \lfloor(\gamma\cdot
2^{264})\rfloor.$$
We started by looking for an integer $e$ such that $z = s^e$, seen as an
element of the number field of $f$, is smooth. After a few core-hours, we
found a value of $e$ such that $z = z_1/z_2$ with $z_1$ and $z_2$
splitting completely
into prime ideals of at most 60 bits. With the lattice-sieving software
of CADO-NFS, we then performed a "special-q descent" for each of these
prime ideals. We remark that one of the prime ideals in $z_1$ was an ideal
of degree 2 above 43, that had to be descended in a specific way,
starting with a polynomial of degree 2 instead of 1. The total time for
descending all the prime ideals was a few minutes. Finally, we found

{\scriptsize
$$\log_G(s) = 431724646474717499532141432099069517832607980262114471597315861099398586114668 \bmod \ell.$$}
Verification scripts in various mathematical software are given in the
NMBRTHRY announcement.

%%\bibliographystyle{alpha}
%%\bibliography{../../biblio-aurore,../../abbrev2,../../crypto,../../biblio,../../others}

\newcommand{\etalchar}[1]{$^{#1}$}
\def\noopsort#1{}\ifx\bibfrench\undefined\def\biling#1#2{#1}\else\def\biling#1%
#2{#2}\fi\def\Inpreparation{\biling{In preparation}{en
  pr{\'e}paration}}\def\Preprint{\biling{Preprint}{pr{\'e}version}}\def\Draft{%
\biling{Draft}{Manuscript}}\def\Toappear{\biling{To appear}{\`A para\^\i
  tre}}\def\Inpress{\biling{In press}{Sous presse}}\def\Seealso{\biling{See
  also}{Voir {\'e}galement}}\def\Editor{\biling{Ed.}{R{\'e}d.}}

\end{document}